\documentclass[12pt]{article}

\newlength{\stefan}
\usepackage{amssymb}
\usepackage{amsmath, amsthm, amsfonts, makeidx}
\usepackage[all]{xy}

\usepackage[usenames,dvipsnames]{color}

\DeclareMathSymbol{\subsetneq}{\mathord}{AMSb}{"26}

\newtheorem{lemma}{Lemma}[section]

\newtheorem{theorem}[lemma]{Theorem}

\newtheorem{proposition}[lemma]{Proposition}
\newtheorem{corollary}[lemma]{Corollary}
\theoremstyle{definition}

\newtheorem{definition}[lemma]{Definition}

\newtheorem{remark}[lemma]{Remark}

\newtheorem{question}[lemma]{Question}

\newcommand{\lp}{\longrightarrow}

\newcommand{\mb}{\mathbb}

\newcommand{\F}{\mb{F}}

\newcommand{\X}{\mathcal{X}}
\newcommand{\XX}{\mathcal{X}}

\newcommand{\C}{\mb{C}}

\newcommand{\Z}{\mb{Z}}
\newcommand{\N}{\mb{N}}

\newcommand{\Aff}{\operatorname{Aff}}

\renewcommand{\ker}{\operatorname{ker}}

\renewcommand{\deg}{\operatorname{deg}}
\newcommand{\Alt}{\operatorname{Alt}}

\newcommand{\Jac}{\operatorname{Jac}}
\newcommand{\GL}{\operatorname{GL}}

\newcommand{\GA}{\operatorname{GA}}

\newcommand{\perm}{\operatorname{Perm}}

\newcommand{\kar}{\operatorname{char}}
\newcommand{\Perm}{\operatorname{Perm}}
\newcommand{\Maps}{\operatorname{Maps}}

\newcommand{\Gal}{\operatorname{Gal}}
\newcommand{\OO}{\mathcal{O}}
\newcommand{\MA}{\operatorname{MA}}
\newcommand{\TA}{\operatorname{TA}}

\newcommand{\Sym}{\operatorname{Sym}}
\newcommand{\BA}{\operatorname{BA}}
\newcommand{\BAs}{\operatorname{BAs}}
\newcommand{\PMA}{\overleftarrow{\MA}}
\newcommand{\PGA}{\overleftarrow{\GA}}
\newcommand{\PTA}{\overleftarrow{\TA}}
\newcommand{\MMA}{\mathcal{M}}

\title{The profinite polynomial automorphism group}
\author{
\begin{tabular}{ll}
Stefan Maubach & Abdul Rauf$\footnote{Supported by DAAD grant ( funding program ID 57076385).}$\\ \small Jacobs University Bremen & \small Jacobs University Bremen\\\small
Bremen, Germany& \small Bremen, Germany \\ \small s.maubach@jacobs-university.de ~~& \small a.rauf@jacobs-university.de
\end{tabular}}

\begin{document}

\maketitle

\begin{abstract}
We introduce an extension of the (tame) polynomial automorphism group over finite fields: the profinite (tame) polynomial automorphism group, which is obtained by putting a natural topology on the automorphism group. We show that most known candidate non-tame automorphisms are inside the profinite tame polynomial automorphism group, giving another result showing that tame maps are potentially ``dense'' inside the set of automorphisms.
We study the profinite tame automorphism group and show that it is not far from the set of bijections obtained by endomorphisms.

\end{abstract}

\tableofcontents

%37P05 iterations of rational or polynomial maps
%11T06 Number theory: polynomials
%12E20 Field theory and polynomials , finite fields
%94A60 Cryptography
%20B20 Permutation groups

\section{Preliminaries}
\label{Section1}

\subsection{Notations and definitions}
The notation $(X_1,\dots,\hat{X}_i,\dots,X_n)$ means that we skip the ith entry.
 $q$ will be a prime power (of $p$, a prime).

\subsection{Introduction}

If $k$ is a field, then in this article, we are considering polynomial maps $F=(F_1,\ldots,F_n)$ where $F_i\in k[X_1,\ldots,X_n]$.
The collection of polynomial maps over $k$ is denoted by $\MA_n(k)$.
They form a monoid under composition $\circ$ (and abelian group under $+$), and each polynomial map indeed induces a map $k^n\lp k^n$.
Thus, in general, we have a map
\[ \MA_n(k)\lp \Maps(k^n,k^n)\]
This map is injective unless $k$ is a finite field - and it's exactly the latter case we'll be discussing in this article.
Thus, we define
\[ \pi_q:\MA_n(\F_q)\lp \Maps(\F_q^n,\F_q^n) \]

The monoid $\MA_n(k)$ has a unit element $(X_1,\ldots,X_n)$. The subset of invertible elements forms a group, and is denoted as
$\GA_n(k)$.  (As a remark, a famous conjecture, the Jacobian Conjecture, states that if $\kar(k)=0$, then $F\in \MA_n(k)$ plus $\det\Jac(F)\in k^*$ implies that $F\in \GA_n(k)$.)
There are a few obviously invertible polynomial maps: \\
(1) Invertible affine maps (i.e $F=TL$ where $T$ is a translation and $L$ is invertible linear). The set of these maps forms a group, denoted by $\Aff_n(k)$.\\
(2) Triangular (or Jonqui\`ere) maps:
$F=(a_1X_1+f_1,\ldots, a_nX_n+f_n)$ where $a_i\in k^*$ and $f_i\in k[X_{i+1},\ldots, X_n]$.
The set of these maps also forms a group, denoted by $\BA_n(k)$. The set of all triangular maps such that $a_i=1$ for $1\leq i\leq n$ is called the set of strictly triangular
polynomial maps, and is a subgroup of $\BA_n(k)$ and is denoted by $\BAs_n(k)$.\\
(See \cite{Essenboek} or any other standard source for proofs on the invertibility of these maps.)\\

It is now natural to define the set of tame automorphisms, $\TA_n(k)=<\BA_n(k), \Aff_n(k)>$.
In dimension 2 it is proven that $\TA_2(k)=\GA_2(k)$ (the Jung-van der Kulk-theorem).
A highly sought-after question (posed by Nagata in 1974) was if $\TA_n(k)=\GA_n(k)$ for some $k$ and $n\geq 3$.
It took about 30 years till Umirbaev and Shestakov proved that if $\kar(k)=0$ then $\TA_3(k)\not = \GA_3(k)$ \cite{Umirbaev-Shestakov,Umirbaev-Shestakov2}.
The problem is still open in higher dimensions and in characteristic $p$ - the latter being the topic of this article.
In fact, one of the motivations for this article is in trying to see if $\pi_{q^m}(\TA_n(\F_q))$ can be different from
$\pi_{q^m}(\GA_n(\F_q))$, which would induce that the groups are not equal and show the existence of wild maps over $\F_q$.

One of the motivating questions for studying this group is the following:
what if $\pi_q(\TA_n(\F_q))$ is unequal to $\pi_q(\GA_n(\F_q))$? Then $\TA_n(\F_q)$ must be unequal to $\GA_n(\F_q)$, and
we have shown that there exist non-tame maps (in a potential trivial way).
In particular, in \cite{Mau03} the following theorem is proven:

\begin{theorem} If $n\geq 2$, then $\pi_q(\TA_n(\F_q))=\Sym(\F_q^n)$ if $q$ is odd or $q=2$.
If $q=2^m$ where $m\geq 2$ then $\pi_q(\TA_n(\F_q))=\Alt(\F_q^n)$.
\end{theorem}

The following natural conjecture was posed in the same paper:

\begin{question} Is there an automorphism $F\in \GA_n(\F_{2^m})$ where $m\geq 2$ such that
$\pi_{2^m}(F)$ is odd?
\end{question}

Such an example would then automatically have to be a non-tame automorphism.
The above question, even though getting reasonable attention, is unsettled for more than ten years now.
But, then the next step is:

\begin{question} Is there an automorphism $F\in \GA_n(\F_q)$ such that $\pi_{q^m}(F)\not \in \pi_{q^m}(\TA_n(\F_q))$?
\end{question}

We address this question for a large class of candidate wild maps in section \ref{NietWild}.

\subsection{Organisation of this paper}

This paper is organised as follows: \\
In section \ref{Section1} (this section) we give a motivation for and overview of the results in the paper. \\
In section \ref{Section2}  we define the profinite endomorphism monoid $\PMA_n(\F_q)$, and the underlying groups
$\PTA_n(\F_q)$ and $\PGA_n(\F_q)$. \\
In section \ref{Section3} we show that a large class of potentially non-tame maps in $\GA_n(\F_q)$ are inside $\PTA_n(\F_q)$ (i.e. are
``profinitely tame'') \\
In section \ref{Section4} we analyze the group of invertible elements in $\PMA_n(\F_q)$, as this is the ``world'' in which $\PGA_n(\F_q)$ and $\PTA_n(\F_q)$ live in. \\
In sections \ref{Section5},\ref{Section5a},\ref{Section6},   we study the ``distance'' between $\PTA_n(\F_q)$ and $\PGA_n(\F_q)$.  Here, section \ref{Section5} (the main bulk) is an unavoidably technical and tricky proof of a bound between
$\pi_{q^m}(\TA_n(\F_q))$ and $\pi_{q^m}(\MA_n(\F_q))\cap \perm( (F_{q^m})^n)$. This bound is made explicit in section \ref{Section6}. \\
In section \ref{Section5a} we give some examples where we compute the actual size of $\pi_{q^m}(\TA_n(\F_q))$ for some specific $q,m$.

\section{Profinite endomorphisms}
\label{Section2}

It is not that hard to see that $\pi_q(\MA_n(\F_q))=\Maps(\F_q^n,\F_q^n)$, i.e. $\pi_q$ is surjective.
It becomes interesting if one wants to study $\pi_{q^m}(\MA_n(\F_q))$, as this will not be equal to $\Maps(\F_{q^m}^n,\F_{q^m}^n)$.
In order to understand this, let us define the group action
\[\Gal(\F_{q^m}:\F_q) \times (\F_{q^m})^n \longrightarrow (\F_{q^m})^n\]  by
\[\phi\cdot(a_1,\dots,a_n)=(\phi(a_1),\dots,\phi(a_n)).\]
If $\alpha\in (\F_{q^m})^n$, we will denote by $[\alpha]$ the orbit $\Gal(\F_{q^m}:\F_q)\alpha$.
It is clear that if $F\in \MA_n(\F_q)$, then $\phi\pi_{q^m}(F)=\pi_{q^m}\phi(F)$. With a little effort one can show that this is the only constraint:

\begin{proposition} \label{smP1.1}
\[ \pi_{q^m}(\MA_n(\F_q))= \{ \sigma\in \Maps(\F_{q^m}^n,\F_{q^m}^n)~|~ \sigma\phi=\phi\sigma ~\forall~ \phi\in \Gal(\F_{q^m}:\F_q)\}.\]
\end{proposition}

The above proposition can be easily proved by the following lemma:

\begin{lemma} \label{def1.4} \label{sm1.2}~\\
(1)
For every $\alpha\in (\F_{q^m})^e$ there exists $f_{\alpha,1}\in \F_q[Y_1,\ldots,Y_e]$ such that
$f_{\alpha,1}(\beta)=0 $ if $[\beta]\not =[\alpha]$ and $f_{\alpha,1}(\beta)=1 $ if $[\beta]=[\alpha]$.\\
(2) In case $\F_{q}(\alpha)=\F_{q^m}$, then for every
$b\in \F_{q^m}$ there exists $f_{\alpha,b}\in \F_q[Y_1,\ldots,Y_e]$ such that
$f_{\alpha,b}(\beta)=0 $ if $[\beta]\not =[\alpha]$ and $f_{\alpha,b}(\beta)=b $ if $\beta =\alpha$.\\
\end{lemma}
\begin{proof}
It is trivial that there exists a polynomial $g_{\alpha,1}\in \F_{q^m}[Y_1,\ldots,Y_e]$ such that  $g_{\alpha,1}(\beta)=0$ unless $\beta=\alpha$, when it is 1.
Defining
\[ f_{\alpha,1}=\prod_{\sigma\in \Gal(\F_{q^m}:\F_q)} \sigma(g_{\alpha,1}) \]
we see that $f_{\alpha,1}\in \F_q[Y_1,\ldots,Y_e]$ as $\sigma(f_{\alpha,1}(\beta))=f_{\alpha,1}(\sigma(\beta))$ for every $\sigma\in \Delta$, $\beta\in  (\F_{q^m})^e$. Now $f_{\alpha,1}(\beta)=0$ if $[\beta] \not= [\alpha]$ and $f_{\alpha,1}(\alpha)=1$, solving (1). Now assuming $\F_{q}[\alpha]=\F_q(\alpha)=\F_{q^m}$, there exists some $h\in \F_q[Y_1,\ldots, Y_e]$ such that $h(\alpha)=b$.
Now $f_{\alpha, b}:=hf_{\alpha,1}$ has the desired property yielding (2).
\end{proof}

\begin{proof} (of proposition \ref{smP1.1})
Let us write $A$ for the right hand side of the equality; we only need to show that $A\subseteq \pi_{q^m}(\MA_n(\F_q))$.
Note that $A$ as well as $\MA_n(\F_q))$ are $\F_q$-vector spaces and $\pi_{q^m}$ is $\F_q$-linear.
If $\alpha\in  \F_{q^m}^n$ and $\sigma\in A$,
   and $\phi$ is a generator of the cyclic group $\Gal(\F_{q^m}:\F_q)$,
then the order of $\sigma(\alpha)$ under $\phi$ must divide the order of $\alpha$ under $\phi$ (as $\phi\sigma(\alpha)=\sigma(\phi(\alpha))$).
Thus $[\F_q(\sigma(\alpha)):\F_q]$ divides $[\F_q(\alpha):\F_q]$ meaning that $\sigma(\alpha)\in \F_q(\alpha)^n$.
If
 $\alpha\in \F_{q^m}^n, \beta\in \F_q(\alpha)^n$ then define $\sigma_{\alpha,\beta}\in A$ as the map which is zero outside of $[\alpha]$, and satisfies $\sigma_{\alpha,\beta}(\alpha)=\beta$. (This fixes an element of $A$.) The maps $\sigma_{\alpha,\beta}$ form a generating set of $A$. Now picking $f_{\alpha,\beta_i}$ from lemma \ref{sm1.2} and forming $F=(f_{\alpha,\beta_1},\ldots,
 f_{\alpha,\beta_n})$ we have $\pi_{q^m}(F)=\sigma_{\alpha,\beta}$ and we are done.
\end{proof}

If $d|m$, then there exists a natural restriction map $\pi_{q^m}(\MA_n(\F_q))\lp \pi_{q^d}(\MA_n(\F_q))$. In a diagram, we get

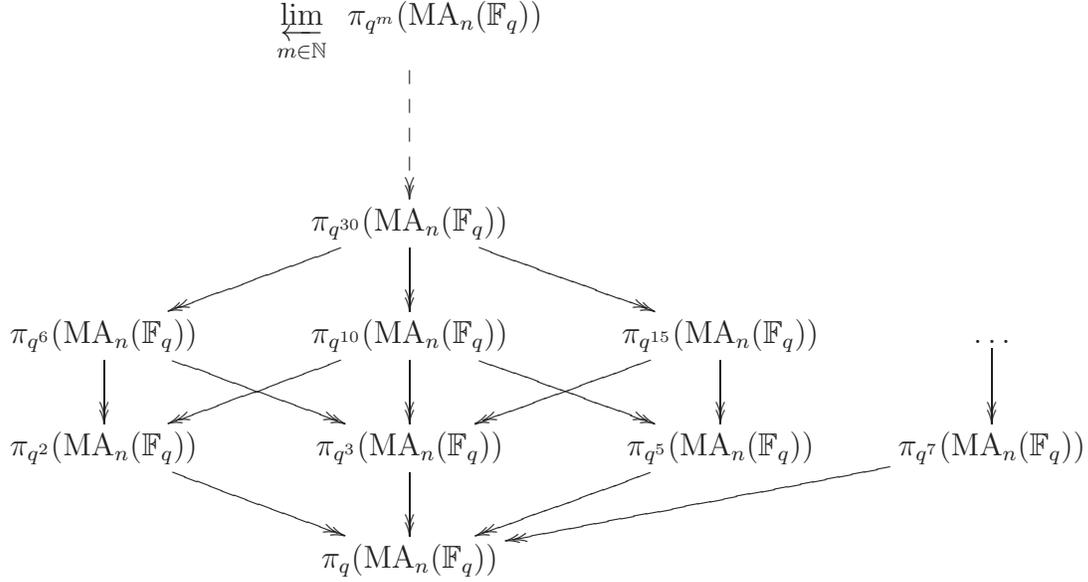
\[
\xymatrix{
& &  \underset{m\in \N}{\varprojlim}~~\pi_{q^m}(\MA_n(\F_q))\ar@{-->>}[dd]\\
&&\\
& &  \pi_{q^{30}}(\MA_n(\F_q))   \ar@{->>}[ld]\ar@{->>}[rd]\ar@{->>}[d]\\
 & \pi_{q^6}(\MA_n(\F_q))  \ar@{->>}[d]\ar@{->>}[rd]& \pi_{q^{10}}(\MA_n(\F_q))  \ar@{->>}[d]\ar@{->>}[rd]\ar@{->>}[ld] &  \pi_{q^{15}}(\MA_n(\F_q))\ar@{->>}[ld]\ar@{->>}[d]&\ldots\ar@{->>}[d]\\
&\pi_{q^2}(\MA_n(\F_q))\ar@{->>}[rd] & \pi_{q^3}(\MA_n(\F_q))\ar@{->>}[d]&  \pi_{q^5}(\MA_n(\F_q))\ar@{->>}[ld]&\pi_{q^7}(\MA_n(\F_q))\ar@{->>}[lld]\\
&& \pi_{q}(\MA_n(\F_q))
} \]

where above we have put the inverse limit $ \underset{m\in \N}{\varprojlim}~~\pi_{q^m}(\MA_n(\F_q))$ of this partially ordered diagram of groups.

\begin{definition} The inverse limit of the above partially ordered diagram of groups is called the profinite polynomial endomorphism monoid over $\F_q$, and denoted as $\PMA_n(\F_q)$. Similarly, we define $\PGA_n(\F_q)$ and
$\PTA_n(\F_q)$ etc.
\end{definition}

Note that $\PMA_n(\F_q)$ can be seen as a subset of $\Maps(\bar{\F}_q^n,\bar{\F}_q^n)$, and in fact proposition
\ref{smP1.1} shows that $\PMA_n(\F_q)=\{ \sigma \in \Maps(\bar{\F}_q^n,\bar{\F}_q^n) ~|~ \sigma\phi=\phi\sigma, \forall \phi\in \Gal(\bar{\F}_q:\F_q)\}$.
We can also embed $\MA_n(\F_q)$ into $\PMA_n(\F_q)$ by the injective map $\bar{\pi}:\MA_n(\F_q)\lp \Maps(\bar{\F}_q^n,\bar{\F}_q^n)$.
Another\footnote{Well, it's actually the definition.} interpretation is that we put a topology on the set $\MA_n(\F_q)$ where a basis of open sets is $\{\ker(\pi_{q^m}); m\in \N\}$, and
$\PMA_n(\F_q)$ is the completion w.r.t. this topology.
(There is similarity with the construction of the $p$-adic integers $\Z_p$ out of the maps $\Z\lp \Z/p^n\Z$, or perhaps better, the construction of the profinite completion of $\Z$, $\hat{\Z}=\prod_p \Z_p$ out of $\Z/n\Z$.)

\section{Automorphisms fixing a variable}
\label{fix}
\label{NietWild}
\label{Section3}

\noindent
{\bf Notations:} If $F\in \GA_n(k[Z])$ and $c\in k$, write $F_c\in \GA_n(k)$ for the restriction of $F$ to $Z=c$.
In case we already have a subscript $F=G_{\sigma}$, then $G_{\sigma,c}=F_c$.
If $F\in \GA_n(k)$, then by $(F,Z)\in \GA_{n+1}(k)$ (or any other appropriate variable in stead of $Z$) we denote the canonical map obtained by
$F$ by adding one dimension. If $F\in \GA_n(k[Z])$ then we identify $F$ on $k[Z]^n$ with $(F,Z)$ on $k^{n+1}$ and denote both by $F$. (In fact, we think of
$\GA_n(k[Z])$ as a subset of $\GA_{n+1}(k)$.)

%One could also state that the lattice $(G_i ~|~ i\in \N^*)$  is a ``$p$-adic approximation of $F$ by elements in $H$'' or something %similar, but the name ``mimickable'' stuck more easily and we prefer it.

The main result of this section is the following proposition:

\begin{theorem} \label{Mimick}
If $F\in \TA_n(\F_q(Z))\cap \GA_n(\F_q[Z])$, and $F_c\in \TA_n(\F_q[c])$ for all $c\in{\F}_{q^m}$,
then $F\in \PTA_n(\F_q[Z])$.
\end{theorem}

This immediately yields the following important corollary:

\begin{corollary} \label{MimickDim2}
\[  \GA_2(\F_q[Z])\subseteq \PTA_2(\F_q[Z])\subseteq \PTA_3(\F_q).\]
\end{corollary}

In particular, it
shows that the famous (notorious?) Nagata automorphism $N=(X-2Y\Delta-Z\Delta^2,Y+Z\Delta,Z)$ where $\Delta=XZ+Y^2$ is
an element of  $\PTA_3(\F_q)$. This shows that Nagata's automorphism is the ``limit'' of tame maps, which calls up resemblance to  \cite{Edo}, where a wild automorphism is shown to be a limit of tame maps in the ``regular'' topology over $\C$.  Slightly more off, the result of Smith (\cite{SMITH}) shows that the Nagata automorphism is stably tame.

Before we state the proof of \ref{Mimick}, we must derive some tools:

\begin{definition}
A map is called {\em strictly Jonqui\`eres} if it is Jonqui\`eres, and has affine part equal to the identity (i.e. linear part identity and zero maps to zero).
\end{definition}

\begin{lemma}\label{remark}
 Let $F\in \TA_n(k(Z))$ be such that the affine part of $F$ is the identity.
Then $F$ can be written as a product of strictly Jonqui\`eres maps and permutations.
\end{lemma}

Proofs like the one below use arguments that can be called standard by those familiar with using the Jung-van der Kulk theorem.
Together with the fact that a precise proof is less insightful and involves even more bookkeeping, we decided to sketch the proof:

\begin{proof} (rough sketch.)
The whole proof works since one can ``push'' elements which are both Jonqui\`eres and affine to one side, since if $E$ is Jonqui\`eres (or affine),
and $D$ is both, then there exists an $E'$ which is Jonqui\`eres (or affine), and $ED=DE'$.
This argument is used to standardize many a decomposition. Here, we emphasize that the final decomposition is by no means of minimal length.\\
(1) First, using the definition of tame maps, we decompose\\ $F=E_1A_1E_2A_2\cdots E_sA_s$ where each $E_i$ is Jonqui\`eres and
each $A_i$ is affine. \\
(2) {\em We may assume that $E_i(0)=A_i(0)=0$ for all $1\leq i\leq s$.} For any pure translation part can be pushed to the left, and then we use the fact that $F(0)=0$. I.e.
the $A_i$ are linear.\\
(3) {\em We may assume that $\det(A_i)=1$ and the $E_i$ are strictly Jonqui\`eres.} To realize this, one must notice that there exists a diagonal linear map $D_i$ satisfying $\det(D_i)=\det(A_i)$, and that we can do this by pushing diagonal linear maps to the left. The result follows
since the determinant of the linear part of $F$ is 1, and hence the determinant of the Jacobian of $F$ is 1.\\
(4) {\em We may assume that each $A_i$ is either diagonal of determinant 1 or -1, or a permutation.}
Now we use Gaussian Elimination to write each $A_i=P_{i1}E_{i1}\cdots P_{it}E_{it}D_{i}$ as a composition of permutations $P_{ij}$,  strictly Jonqui\`eres (elementary linear) maps $E_{ij}$, and one  diagonal map $D_{i}$. We may assume that each $E_{ij}$ is in fact upper triangular, by  conjugating with a permutation. Note that the determinant of each $E_{ij}$ is 1, and of each $P_{ij}$ is 1 or -1, so the determinant of $D_i$ is 1 or -1.\\
(5) {\em We may assume that each $A_i$ is a permutation.} We have to replace the diagonal linear maps $D_i\in \GL_n(k(Z))$ which have determinant 1 or -1. First, write the diagonal linear map $D_i=D_{i1}\cdots D_{it}\tilde{D}_i$, where the $D_{ij}$ are diagonal linear of determinant 1  that have 1's on n-2 places, and at most two diagonal elements which are not 1, and $\tilde{D}_i$ has 1 on the diagonal except at one place, where it is 1 or -1.
The following formulas explains how to write a diagonal map $D_{it}$ as product of (linear) strictly Jonqui\`eres maps and permutations, as well as $\tilde{D}_i$:
{
\[
\begin{array}{l}
P:=\left( \begin{array}{cc}
0& 1\\
1 & 0\\
\end{array} \right),~ E[a]:=\left( \begin{array}{cc}
1& a\\
0 & 1\\
\end{array} \right)\\

\left( \begin{array}{cc}
f^{-1}& 0\\
0 & f\\
\end{array} \right)
=
E[f^{-1}]\cdot
P\cdot
E[1-f]\cdot
P\cdot
E[-1]\cdot
P\cdot
E[1-f^{-1}]\cdot P,\\
\left( \begin{array}{cc}
1& 0\\
0 & -1\\
\end{array} \right)=
E[1]\cdot P\cdot E[-1]\cdot P\cdot E[1]\cdot P.
\end{array}\]}
This finishes  (the rough sketch of) the proof.
\end{proof}

\begin{lemma} \label{ClosedSetMimick} \label{OpenSetMimick}
Let $g(Z)\in \F_q[Z]$.\\
Let $F\in \TA_n(\F_q[Z,g(Z)^{-1}])\cap \GA_n(\F_q[Z])$ be such that the affine part is the identity.
Assume that
$F_c\in \TA_n(\F_q[c])$ for all $c\in{\F}_{q^m}$ and all $m$.
Then for any $m\in \N^*$ we find two maps $G,\tilde{G} \in \TA_n(\F_q[Z])$ such that for $c\in \F_{q^m}$:\\
(1) $G_{c}=I_n$ if $g(c)\not =0$,\\
(2) $G_{c}=F_{c}$ if $g(c)=0$,\\
(3) $\tilde{G}_{c}=F_{c}$ if $g(c)\not =0$.\\
\end{lemma}

\begin{proof} We may assume that $m$ is such that $g$ factors completely into linear factors over $\F_{q^m}$ (for the result for divisors of $m$ is implied by the result for $m$).
Let $\alpha$ be a root of $g$.
Consider $F_{\alpha}$, which by assumption and remark \ref{remark}
can be written as a composition of strictly Jonqui\`eres maps $e_i$ and permutations $p_i$ : $F_{\alpha}=e_1p_1e_2\ldots e_sp_s$.
Write the $e_i$ as $I_n+H_i$ where $H_i$ is strictly upper triangular.
We can even write $H_i=f_i(\alpha,\XX)$ where $f_i(Z,\XX) \in \F_q[Z,\XX]^n$ (and $\XX$ stands for $X_1,\ldots,X_n$).

Now we define $\rho:=1-g^{q^m-1}\in \F_q[Z]$ and $E_i:=I_n+ \rho f_i(Z, \XX)$.
Note that all $E_i \in \TA_n(\F_q[Z])$.
We define $G:=E_1p_1E_2p_2\cdots E_sp_s$. Our claim is that this map acts as required.  Since for $c\in \F_{q^m}$ we have
$\rho(c)=0$ if and only if $g(c)\not =0$, it follows that in that case
$G_c=I_n$.
Since $E_{i,\alpha}=e_i$  by construction, we have $G_{\alpha}=F_{\alpha}$.
Now let $\Phi$ be an element of the galois group $Gal(\F_{q^m}:\F_{q})$. The remaining question is if $G_{\Phi(\alpha)}=F_{\Phi(\alpha)}$.
Note that if $P(\XX,Z)\in \F_q[\XX,Z]$ then $\Phi(P(\alpha,\XX))=P(\Phi(\alpha),\XX)$. This implies that if $F\in \GA_n(\F_q[Z])$, then
$F_{\Phi(\alpha)}=\Phi(F_{\alpha})$. Thus $G_{\Phi(\alpha)}=\Phi(G_{\alpha})=\Phi(F_{\alpha})=F_{\Phi(\alpha)}$ and
we are done.

In order to construct $\tilde{G}$, we again consider the decomposition\\ $F=e_1p_1e_2p_2\cdots e_sp_s$ where the  $e_i$.
Now we modify $H_i$ in $e_i:=I_{n+1}+H_i$ in the following way:
replace each fraction $g^{-t}$ by $g^{t(q^m-2)}$, making new elements
$\tilde{H}_i$ which are in $\MA_n(\F_q[Z])$. Write $E_i:=I_{n}+\tilde{H}_i$, and
define $G:=E_1p_1E_2p_2\cdots E_sp_s$.
Note that if $c\in \F_{q^m}$ and $g(c)\not =0$,
then $g^{-t}(c)=g^{t(q^m-2)}(c)$, thus also $E_{i,c}=e_{i,c}$.
In fact, the latter remark implies that $G_c=F_c$ for all $c\in \F_{q^m}$ such that $g(c)\not =0$.
\end{proof}

\begin{proof} (of proposition \ref{Mimick})\\
{\bf (1)} We may assume that the affine part of $F$ is the identity, by, if necessary, composing with a suitable affine map.\\
{\bf (2)} Since $F\in \TA_n(\F_q(Z))$, we can use Lemma \ref{remark} and decompose $F$ into strictly Jonqui\`eres maps over $\F_q(Z)$ and permutations.
Gathering all denominators which appear in this decomposition, we can assume that $F\in \TA_n(\F_q[Z,g(Z)^{-1}])$ for some
$g(Z)$. We can assume $m$ to be such that $g$ factors into linear parts over $\F_{q^m}$.\\
{\bf (3)} We may assume that if $g(c)\not =0$ then $F_c=I$:
Using lemma \ref{OpenSetMimick} we can find $\tilde{G}\in \TA_n(\F_q[Z])$ such that $\tilde{G}_c=F_c$ if $g(c)\not =0$. We can replace $F$ by $\tilde{G}^{-1}F$.\\
{\bf (4)} Using lemma \ref{OpenSetMimick} we find $G\in \TA_n(\F_{q}[Z])$ such that $F_{c}=G_c$ if $g(c)\not =0$, $F_c=G_c=I$ if $g(c)=0$. Thus, $G^{-1}F$ is the identity map on $\F_{q^m}^n$.\\
Thus, for every $m\in \N$ we can find a $G_m\in \TA_n(\F_q[Z])$ such that $\pi_{q^m}(G_m)=\pi_{q^m}(F)$ (as permutation on
$\F_{q^m}^{n+1}$), meaning that $F\in \PTA_n(\F_q[Z])\subset \PTA_{n+1}(\F_q)$.
\end{proof}

To finish this section, the below lemma clarifies exactly when a map in $\PTA_n(\F_q)\cap \GA_n(\F_q)$ is in $\TA_n(\F_q)$.

\begin{lemma} Let $F\in \PTA_n(\F_q)$, and suppose that there exists a bound $d\in \N$ such that for all $m\in \N$ we have
 $T_m\in \TA_n(\F_{q^m})$ such that $\pi_{q^m}(T_m)=\pi_{q^m}(F)$ and $\deg(T_m)\leq d$. Then $F\in \TA_n(\F_q)$.
\end{lemma}

\begin{proof}
There exists at most one $T\in \TA_n(\F_q)$ of degree $<q^m$ such that $\pi_{q^m}(T)=\pi_{q^m}(F)$. This means that the sequence  $T_1,T_2,T_3,\ldots$ must stabilize. Thus $F=T_m$ if $q^m>d$.
\end{proof}

Note that the converse of the above lemma is trivially true.

\section{The profinite permutations induced by endomorphisms}
\label{Section4}

\begin{definition} Define $\MMA_n(\F_q):=\PMA_n(\F_q)\cap \perm(\bar{\F}_q^n)$.
\end{definition}

$\MMA_n(\F_q)$ is the set of permutations induced by {\em en}domorphisms. It is equal to the set of permutation of $\bar{\F}_q^n$ which commute with $\Gal(\bar{\F}_q:\F_q)$. We do have the following inclusions:

\[ \PTA_n(\F_q)\subseteq \PGA_n(\F_q)\subseteq\MMA_n(\F_q).\]

This means that we have to briefly analyze $\MMA_n(\F_q)$ as this apparently is the ``world'' in which our more complicated objects $\PTA_n(\F_q)$ and $\PGA_n(\F_q)$ live.

\begin{definition} Define $\X_d$ as the union of all orbits of size $d$ of the action of $\Gal(\bar{\F}_q:\F_q)$ on  $\bar{\F}_q^n$.
\end{definition}

If we have $\sigma, \rho\in \Gal(\bar{\F}_q:\F_q)$, then in this section we denote the action of $\sigma$ on $x\in X$ by $x^{\sigma}$, and thus  $(x)^{\sigma\rho} = ((x)^\sigma)^\rho$. ) This is for now more convenient than the notation $\sigma(x)$, as in the below corollary we do not need to talk about the opposite group action etc.

\begin{corollary} (of \ref{smP1.1})\\\label{cor3.1}
(1)
\[ \MMA_n(\F_q)=\bigoplus_{d\in \N}  G_d\]
where $G_d$ is the set of permutations on $\X_d$ which commute with $\Gal(\F_{q^d}:\F_q)$. \\
(2) This means that
\[ G_d\cong \Z/d\Z \wr \perm(r_d)\]
where $r_d$ is the amount of orbits of size $d$ (i.e. $r_d=d^{-1}\#\X_d$).\\
(3)

This is equivalent to the following statement:  $$G_d\cong ((\Z/d\Z)^{r_d}\rtimes_h \Perm(r_d)$$ where $h:\Perm(r_d)\lp Aut(({\Z/d\Z})^{r_d})$ such that $h(\sigma)(a_1,a_2,\dots,a_{r_d})=(a_{(1)^\sigma},a_{(2)^\sigma},\dots,a_{{(r_d)}^\sigma}).$
\end{corollary}

\begin{proof} Part (1)  is obvious using \ref{smP1.1}.\\
 Let $\OO_1,\ldots, \OO_{r_d}$ be the disjoint orbits in $\X_d$.
Let $\alpha_i\in \OO_i$ be chosen from the orbits (at random). Write $\Gal(\F_{q^d}:\F_q)=<\phi>\cong \Z/d\Z$.
Now if a permutation $P$ of $\X_d$ commutes with  $\Gal(\F_{q^d}:\F_q)$,
then it permutes the orbits; say $\sigma\in \perm(r_d)$ is this permutation. It sends $[\alpha_i]$ to $
[\alpha_{(i)^\sigma}]$ and thus $(\alpha_i)^P=\phi^{a_i}\alpha_{(i)^\sigma }$ where $a_i\in \Z/d\Z$.
Thus, as a set we have equality. We only need to check how the multiplication acts:

We can construct a map
\[  \begin{array}{rcl}
\varphi:G_d&\lp&  (\Z/d\Z)^{r_d}\rtimes_h (\Perm(r_d)   \\
P&\lp& (a_1,a_2,\dots, a_{r_d}; \sigma_{P}).
\end{array}\]
Write $\varphi(P)=(a_1,\ldots,a_n;\sigma)$ and $\varphi(Q)=(b_1,\ldots,b_n;\rho)$ for some $P,Q\in G_d$.  Now
$(\alpha_i)PQ=(\phi^{a_i}\alpha_{(i)^\sigma})Q=\phi^{a_i}(\alpha_{(i)^\sigma})Q
=\phi^{a_i}\phi^{b_{(i)^\sigma}}\alpha_{((i)^\sigma)^\rho}=\phi^{{a_i}+{b_{(i)^\sigma}}}\alpha_{(i)^{\sigma\rho}},$
thus
\[\varphi(PQ)=(a_1+b_{(1)^\sigma},\ldots,a_n+b_{(n)^\sigma}; \sigma \rho).\]
and since we want $\varphi(P)\cdot \varphi(Q)=\varphi(PQ)$ we get the multiplication rule for semidirect product
\[ (a_1,\ldots,a_n;\sigma)\cdot (b_1,\ldots,b_n;\rho) =(a_1+b_{(1)^\sigma},\ldots,a_n+b_{(n)^\sigma}; \sigma \rho)\]
%The multiplication $(b_1,\ldots,b_n;\rho)\cdot (a_1,\ldots,a_n;\sigma)=(b_1+a_{\rho(1)},\ldots,b_n+a_{\rho(n)}; \rho\sigma)$
%would be a multiplication of a semidirect product - so the above is an opposite multiplication.
%Applying the opposite multiplication has reverted the order $\sigma\rho\lp \rho\sigma$ though, so we need to have the opposite multiplication inside $\perm(r_d)$.

Thus $\varphi$ is a group homomorphism.
Since $Ker{(\varphi)}=\{P\in G_d: \varphi(P)=(0,0,\dots,0,id)\}=\{id_{G_d}\},$
the map is injective.
Since the orders of the groups are the same, we have an isomorphism.
\end{proof}

\section{The profinite tame automorphism group acting on orbits of a fixed size}
\label{Section5}

Our goal is to understand ``how far'' $\PTA_n(\F_q)$ can be from $\PGA_n(\F_q)$ if $n\geq 3$. Since $\GA_n(\F_q)$ is quite ungraspeable if $n\geq 3$, we are in fact trying to understand ``how far''
$\PTA_n(\F_q)$ is from $\MMA_n(\F_q)$, which squeezes in $\PGA_n(\F_q)$ in between of them.

In most of this section, we fix $m\in \N$. We have to introduce a score of notations:
\begin{itemize}
\item  $\Delta=\Gal(\F_{q^m}:\F_q)$,
\item
$\X$ is the union of the orbits of size $m$ in $\bar{\F}_q^n$,
\item  $\Omega$ is the set of orbits of size $m$ in
 $\bar{\F}_q^{n-1}$,
 \item  $\bar{\X}$ is  the quotient of $\X$ under $\Delta$, i.e. the set of orbits of size $m$,
 \item set $q=p^l$: $p$ is prime integer, $m$, $l\geq 1$,
 \item we fix the dimension  $n\geq3$.
 \end{itemize}

 The action of $\TA_n(\F_q)$ on $\bar{\F}_q^n$ restricts naturally to $\X$, and also induces an action on $\bar{\X}$.
 This means that we have a natural group homomorphism $\TA_n(\F_q)\lp \perm(\bar{\X})$.
 We denote by \texttt{G} the image of this group homomorphism.

Our first  goal is to prove the below theorem.
\begin{theorem}\label{Main}
For $n\geq3$, we have $\Alt(\bar{\X})\subseteq\texttt{G}.$
\end{theorem}

The rest of this section is devoted to proving this theorem. The generic outline of the proof is as follows:
It is enough to show that \texttt{G} is primitive by a theorem of Jordan:

\begin{theorem}\label{Jordan}(Jordan)
Let $G$ be a primitive subgroup of $S_n$. Suppose $G$ contains a 3-cycle. Then $G$ contains the alternating subgroup $A_n$.
\end{theorem}

{\bf Overview of the proof of theorem \ref{Main}.}\\
We will prove that the group \texttt{G} contains a 3-cycle (lemma \ref{3-cycles}) and is 2-transitive  (which implies primitive). The latter is the most complicated  part, requiring some delicate induction arguments: lemmas \ref{0.6.10}-\ref{lemma1.10} are preparations to prove  proposition \ref{1.11} (2-transitivity for a large class of points). Along with this, lemmas \ref{1.14}-\ref{1.1818} are
preparations to prove proposition \ref{1.725} (2-transitivity in general).
The difficulty in this proof are obviously its length and occasional technicality, but also in the rather complicated induction (the lemmas \ref{1.14}-\ref{1.1818}) which makes the proof quite nontrivial.

\begin{definition}
(i) Let $a\in \Omega$. Then define $$ \tau_{a,1}:=(X_1+f_{a,1}(X_2,\dots,X_n), X_2, \dots, X_n)$$ (ii) Let $i\in\{2, \dots, n\}$. Define
$$\tau_i:=(X_i, X_2,\dots,X_{i-1},X_1,X_{i+1},\dots,X_n),$$ the map interchanging $X_i$ and $X_1$.
\end{definition}

\begin{definition}
In this article we define the lexicographic ordering of two vectors ${u}:=(u_1, \dots, u_n), {v}:=(v_1, \dots, v_n)$, notation $u\geq_{Lex} v$, if there exists $m\in \N, 0\leq m\leq n$ such that
$u_m>v_m, (u_{m+1}, u_{m+2},\ldots, u_n)=(v_{m+1},v_{m+2},\ldots,v_n)$. (i.e. the weight is at the ``head'' of a vector, not the tail).
\end{definition}

%IS IT POSSIBLE TO CONTRACT/SHORTEN THE 2 UPCOMING  EXAMPLES TOGETHER WITH THE BELOW LEMMA?

\begin{lemma} \label{3-cycles}

The group \texttt{G} contains 3-cycles
\end{lemma}
\begin{proof} Consider $\F_{q^m}=\F_q(t)$. Let $a:=(0,\dots,0,t)\in \Omega$ and let \[s:=\tau_{a,1}=(X_1+g_{a,1}(X_2,\dots,X_n), X_2, \dots, X_n),\]
\[d:=\tau_2\tau_{a,1}\tau_2=(X_1, X_2+g_{a,1}(X_1,X_3,\dots,X_n), X_3,\dots,X_n).\]

Consider $L_1=\{[(a_1,0,0,\dots,0,a_n)]|a_n\in [t]$ and $a_1\in \F_{q^m}\}$ and $L_2=\{[(0,a_2,0,\dots,0,a_n)]|a_2\in \F_{q^m}$ and $a_n\in [t]\}$, where $L_1$, $L_2$ are subsets of $\bar\X$ and $[t]=\{\phi(t) : \phi \in \Delta\}$.
Then $s$ permutes only the set $L_1$ and $d$ permutes only the set $L_2$. Both $s$ and $d$ are cyclic of order $p=\kar(\F_{q^m})$ on $L_1$ and $L_2$ respectively. Let $w=s^{-1}d^{-1}sd$. Then $w$ acts trivially on $\bar\X\setminus(L_1\bigcup L_2)$ and nontrivially only on a subset of $L_1\bigcup L_2$. Now if $b\notin L_2$ and $s(b)\notin L_2$,  then since d works only on elements of $L_2$ one can check easily that $w(b)=b$ . Similarly if $b\notin L_1$ and $d(b)\notin L_1$ then $w(b)=b$. The other cases include:\\
{\em (1)} $b\notin L_2$ and $s(b)\in L_2$ (the element$ D:=[(-1,0,\dots,0,t)])$,\\
{\em (2)} $b\notin L_1$ and $d(b)\in L_1$ (the element $E:= [(0,-1,0\dots,0,t)])$,\\
{\em (3)} $b\in L_1$ and $d(b)\in L_2$ $(F:=[(0,\dots,0,t)])$.\\
Since $s(D)=F$,$s(E)=E$,$d(E)=F$,$d(D)=D$,$s(F)\notin L_2$,$d(F)\notin L_1$. Using this observation we see that $w(D)=E,w(E)=F,w(F)=D$ and $w$ is the required $3$-cycle. \end{proof}

Here starts the technical proof of 2-transitivity (proposition \ref{1.725}).
We slowly start by connecting more and more pairs $([r],[u])$ and $([s],[v])$ where $[r],[s],[v],[u]\in \bar{X}$. (For most lemmas, $[v]=[u]$.)
We introduce the following definition:

\begin{definition} We say that $u,v\in \F_{q^m}^n$ are weakly conjugate (or, $u$ is a weak conjugate of $v$) if
$[u_i]=[v_i]$ for all $1\leq i\leq n$. We denote this as $u\approx v$.
We also denote $u\not \approx v$ for ``not $u\approx v$''.
\end{definition}

Note that $[u]=[v]$ implies $u\approx v$, but not the other way around. Similarly, $u\not \approx v$ implies $[u]\neq [v]$. In the below lemmas we will have to give special attention to cases where a  pair of $r,s,u$ is weakly conjugate, as this complicates things.

\begin{lemma}\label{0.6.10}
Let $r,u\in \X$ such that $[r]\not =[u]$. Assume $\F_{q^m}=\F_q(r_i)$ and $[r_j]\neq[u_j]$ for some $i,j\in\{1,2,\dots,n\}$. Then for any $k\in\{1,2.\dots,n\}$ with $k\neq i,j$ there exist $P\in TA_n(\F_q)$ such that $P[r]=[(r_1,\dots,r_{k-1},v,r_{k+1},\dots,r_n)]$ for any $v\in\F_{q^m}$ and $P[u]=[u]$.
\end{lemma}
\begin{proof}
%Fix $a=(r_1,\dots,r_{k-1},\tilde{r}_k,r_{k+1},\dots,r_n)$ and define a polynomial map $p\in\F_q[X_1,\dots,\widehat{X_k},\dots,X_n]$ such that $p(a)=v-r_k$, $v\in\F_{q^m}$.
Fix $a=(r_1,\dots,\hat{r}_k,\dots,r_n).$ Define $P:=(X_1,\dots,X_k+f_{a,v-r_k}(X_1,\dots,\widehat{X_k},\dots,X_n),\dots,X_n)$, where $f_{a,v-r_k}$ is as in lemma \ref{def1.4}. (Note that since $r_i$ generates $\F_{q^m}$, we can indeed apply this lemma as $a=(\dots,r_i,\dots).$)  Hence we have $P([r])=[(r_1,\dots,r_{k-1},v,r_{k+1},\dots,r_k)]$ and $P([u])=[u]$.
\end{proof}
%\begin{remark}\label{0.6.10}
% Let $r,u\in \X$ s.t. $r=(r_1,\dots, r_n)$ and $u=(u_1,\dots,u_n)$  are in different orbits with $\F_{q^m}=\F_q(r_i)$ and $[r_j]\neq[u_j]$ for $i,j\in\{1,2.\dots,n\}$. Then for any $k\in\{1,2.\dots,n\}$ with $k\neq i,j$ we can define the polynomial map $p\in\F_q[X_1,\dots,\widehat{X_k},\dots,X_n]$ such that $p(r_1,\dots,\widehat{r_k},\dots,r_n)=v-r_k$ for any $v\in\F_{q^m}$.
%For any $a\in\Omega$ define the tame map $P=(X_1,\dots,X_k+h_{((a),p(a))}(X_1,\dots,\widehat{X_k},\dots,X_n),\dots,X_n)$, then we get $P([r])=[(r_1,\dots,r_{k-1},v,r_{k+1},\dots,r_k)]$ and $P([u])=[u]$.
%\end{remark}
We will use this lemma in the proof of following lemma several times. The below lemma is elementary but lenghty.
\begin{lemma} \label{1.61}
Let $s,r,u\in \X$ s.t. $s=(s_1,\dots, s_n)$, $r=(r_1,\dots, r_n)$, and $u=(u_1,\dots,u_n)$  be in different orbits with $\F_{q^m}=\F_q(r_1)=\F_q(s_1)$ and suppose $r\not \approx u, s\not \approx u$. Then there exist $F\in TA_n(\F_q)$ s.t. $F([r])=[s]$ and $F([u])=[u]$.
\end{lemma}
\begin{proof}
We divided our lemma into the following five cases.\\
{Case 1}. $[r_1]\neq [u_1]$ and $[s_1]\neq [u_1]$.\\
{Case 2}. $[r_1]\neq [u_1]$ and $[s_i]\neq [u_i]$ for some $2\leq i\leq n$.\\
{Case 3}. $[s_1]\neq [u_1]$ and $[r_i]\neq [u_i]$ for some $2\leq i\leq n$.\\
{Case 4}. $[r_i]\neq [u_i]$ and $[s_j]\neq [u_j]$ for some $i,j\in \{2,3,\dots,n\}.$

\textbf{Case 1}. $[r_1]\neq [u_1]$ and $[s_1]\neq [u_1]$.\\
As $\F_{q^m}=\F_q(r_1)$ and $[r_1]\neq [u_1]$ so by applying the tame map as in lemma \ref{0.6.10} several times we map the orbit $[r]$ to $[r_1,s_1,s_3\dots,s_n]$ and $[u]$ remains unchanged.\\
\textsf{Subcase 1.1}. If $[(s_1,s_3,\dots,s_n)]\neq[(u_2,\dots,u_n)].$ Fix $a:=(s_1,s_3,\dots,s_n)$. Define$P_1:=(X_1+f_{a,s_1-r_1}(X_2,\ldots, X_n), X_2,\ldots,X_n)$ where $f_{a,s_1-r_1}$ is as in lemma \ref{def1.4}. (Note that since $s_1$ generates $F_{q^m}$, we can indeed apply this lemma as $a=(s_1,\ldots)$.) Thus $P_1[(r_1,s_1,s_3,\dots,s_n)]=[(s_1,s_1,s_3,\dots,s_n)]$ and $P_1[u]=[u].$ Now as $\F_{q^m}=\F_q(s_1)$ and $[s_1]\neq [u_1]$ so by applying the tame map as in lemma \ref{0.6.10}  we map the orbit $[(s_1,s_1,s_3,\dots,s_n)]$ to $[(s_1,s_2,s_3,\dots,s_n)]$ and $[u]$ remains unchanged.\\
\textsf{Subcase 1.2.} If $[(s_1,s_3,\dots,s_n)]=[(u_2,\dots,u_n)].$ In this case we have $s_1\in[u_2]$ and hence $u_2$ is also a generator of $\F_{q^m}$.
Using lemma \ref{0.6.10} and $[r_1]\neq[u_1]$ we can send $[u]$ to $[u_1,u_2,u_1,u_4,\dots,u_n]$ and $[(r_1,s_1,s_3,\dots,s_n)]$ remains unchanged.
Now with $r_1$ generator of $\F_{q^m}$ and $[r_1]\neq[u_1]$ send $[(r_1,s_1,s_3,\dots,s_n)]$ to $[(r_1,s_1,s_1,s_4,\dots,s_n)]$ and $[u_1,u_2,u_1,u_4,\dots,u_n]$ remains unchanged by lemma \ref{0.6.10}. With $s_1$ generator of $\F_{q^m}$ and $[s_1]\neq[u_1]$ send $[(r_1,s_1,s_1,s_4,\dots,s_n)]$ to $[(s_1,s_2,s_1,s_4,\dots,s_n)]$ and $[u_1,u_2,u_1,u_4,\dots,u_n]$ remains unchanged by applying lemma \ref{0.6.10} two times.  With $s_1$ generator of $\F_{q^m}$ and $[s_1]\neq[u_1]$ send $[(s_1,s_2,s_1,s_4,\dots,s_n)]$ to $[(s_1,s_2,s_3,s_4,\dots,s_n)]$ and $[u_1,u_2,u_1,u_4,\dots,u_n]$ remains unchanged by applying lemma \ref{0.6.10}.
 With $u_2$ generator of $\F_{q^m}$ and $[s_1]\neq[u_1]$ send $[u_1,u_2,u_1,u_4,\dots,u_n]$ to $[u]$ and $[s]$ remains unchanged by applying lemma \ref{0.6.10}.

 \textbf{Case 2}. $[r_1]\neq[u_1]$ and $[s_i]\neq[u_i]$ for $2\leq i\leq n$.\\
  Since $n\geq 3$ there exist $k\in\{2,3,\dots,n\}$ such that $k\neq i$. Without loss of generality (simplifying notations) we assume $i=2, k=3$. With $r_1$ generator of $\F_{q^m}$ and $[r_1]\neq[u_1]$, map $[r]$ to $[(r_1,s_2,s_1,s_4,\dots,s_n)]$ and $[u]$ remains unchanged by applying lemma \ref{0.6.10} several time. With $[s_2]\neq[u_2]$ and $s_1$ generator of $\F_{q^m}$, map $[(r_1,s_2,s_1,s_4,\dots,s_n)]$ to $[(s_1,s_2,s_1,s_4,\dots,s_n)]$ and $[u]$ remains unchanged by lemma \ref{0.6.10}.
 With $[s_2]\neq[u_2]$ and $s_1$ generator of $\F_{q^m}$, map\\
  $[(s_1,s_2,s_1,s_4,\dots,s_n)]$ to $[(s_1,s_2,s_3,s_4,\dots,s_n)]$= $[s]$ and $[u]$ remains unchanged by lemma \ref{0.6.10}.

 \textbf{Case 3}.  $[s_1]\neq[u_1]$ and $[r_i]\neq[u_i]$ for $2\leq i\leq n$.\\
  By case 2 there exist a tame map $F\in TA_n(\F_q)$ such that $F[s]=[r]$ and $F[u]=[u]$, hence $F^{-1}[r]=[s]$ and $F^{-1}[u]=[u].$

 \textbf{Case 4}. $[r_i]\neq[u_i]$ and $[s_j]\neq[u_j]$ for some $i,j\in\{2,3,\dots,n\}$.\\
 \textsf{Subcase 4.1}. $[r_i]\neq[u_i]$ and $[s_j]\neq[u_j]$ for some $i,j\in\{2,3,\dots,n\}$ with $i\neq j.$\\
 Without loss of generality (simplifying notations) we assume $i=2, j=3$.
With $\F_{q^m}=\F_q(r_1)$ and $[r_2]\neq [u_2]$,  send the orbit $[r]$ to $[(r_1,r_2,s_3,s_4,\dots,s_n)]$ and the orbit $[u]$ remains unchanged by applying the lemma \ref{0.6.10} several times. With $\F_{q^m}=\F_q(r_1)$ and $[s_3]\neq [u_3]$, map $[(r_1,r_2,s_3,s_4,\dots,s_n)]$ to $[(r_1,s_1,s_3,s_4,\dots,s_n)]$ and $[u]$ remains unchanged by lemma \ref{0.6.10}. With $\F_{q^m}=\F_q(s_1)$ and $[s_3]\neq [u_3]$, map $[(r_1,s_1,s_3,s_4,\dots,s_n)]$ to $[(s_1,s_1,s_3,s_4,\dots,s_n)]$ and $[u]$ remains unchanged by lemma \ref{0.6.10}. With $\F_{q^m}=\F_q(s_1)$ and $[s_3]\neq [u_3]$, map $[(s_1,s_1,s_3,s_4,\dots,s_n)]$ to $[(s_1,s_2,s_3,s_4,\dots,s_n)]$=$[s]$ and the orbit $[u]$ remains unchanged by lemma \ref{0.6.10}.\\
\textsf{Subcase 4.2}. If $[r_i]\neq[u_i]$ and $[s_i]\neq[u_i]$ for some $i\in\{2,3,\dots,n\}$.\\
Since $n\geq 3$ there exist $k\in\{2,3,\dots,n\}$ such that $k\neq i$. Without loss of generality (simplifying notations) we assume $i=2, k=3$.
 As $\F_{q^m}=\F_q(r_1)$ and $[r_2]\neq [u_2]$, so by applying the tame map as in lemma \ref{0.6.10} several times we map the orbit $[r]$ to\\ $[(r_1,r_2,s_1,s_4,s_5,\dots,s_n)]$ and the orbit $[u]$ remains unchanged (Note that here we send $r_3$ to $s_1$, because $s_1$ is generator of $\F_{q^m}$ and we need this to send $r_1$ to $s_1$.).
With $\F_{q^m}=\F_q(s_1)$ and $[r_2]\neq [u_2]$ we map the orbit $[(r_1,r_2,s_1,s_4,s_5,\dots,s_n)]$ to $[(s_1,r_2,s_1,s_4,s_5,\dots,s_n)]$ and the orbit $[u]$ remains unchanged by lemma \ref{0.6.10}.
With $\F_{q^m}=\F_q(s_1)$ and $[r_2]\neq [u_2]$ we map the orbit $[(s_1,r_2,s_1,s_4,s_5,\dots,s_n)]$ to  $[(s_1,r_2,s_3,s_4,s_5,\dots,s_n)]$ and the orbit $[u]$ remains unchanged by lemma \ref{0.6.10}.\\
\textsf{Subsubcase 4.2.1}\\
 If $[(s_1,\hat{r}_2,s_3,s_4,\dots,s_n)]\neq[u_1,\hat{u}_2,u_3,u_4,\dots,u_n]$. Fix $a=(s_1,\hat{r}_2,s_3,s_4,\dots,s_n)$. Define a tame map $P=(X_1,X_2+h_{a,s_2-r_2}(X_1,\hat{X}_2,X_3,\dots,X_n),X_3,X_4,\dots,X_n)$ where $h_{a,s-2-r_2}$ is as in lemma \ref{def1.4}. (Note that since $s_1$ generates $\F_{q^m}=\F_q(s_1)$, we can indeed apply lemma \ref{def1.4} as $a=(s_1,\dots)$.)
  Thus $P[(s_1,r_2,s_3,s_4,\dots,s_n)]=[(s_1,s_2,\dots,s_n)]=[s]$ and $P[u]=[u].$\\
\textsf{Subsubcase 4.2.2}\\
If $[(s_1,\hat{r}_2,s_3,s_4,\dots,s_n)]\neq[u_1,\hat{u}_2,u_3,u_4,\dots,u_n]$. In this case $u_1\in[s_1]$, so $\F_q(u_1)=\F_{q^m}$. Thus using $[r_2]\neq[u_2]$, send $[u]$ to $[(u_1,u_2,u_2,u_4,u_5,\dots,u_n)]$ and $[(s_1,r_2,s_3,s_4,\dots,s_n)]$ remains unchanged by lemma \ref{0.6.10}. Now using  $\F_q(s_1)=\F_{q^m}$ and $[r_2]\neq[u_2]$, send $[(s_1,r_2,s_3,s_4,\dots,s_n)]$ to $[(s_1,r_2,s_2,s_4,s_5,\dots,s_n)]$ and $[(u_1,u_2,u_2,u_4,u_5,\dots,u_n)]$ remains unchanged by lemma \ref{0.6.10}. With  $\F_q(s_1)=\F_{q^m}$ and $[s_2]\neq[u_2]$, send $[(s_1,r_2,s_2,s_4,s_5,\dots,s_n)]$ to $[(s_1,s_2,s_2,s_4,s_5,\dots,s_n)]$ and $[(u_1,u_2,u_2,u_4,u_5,\dots,u_n)]$ remains unchanged by lemma \ref{0.6.10}. With  $\F_q(s_1)=\F_{q^m}$ and $[s_2]\neq[u_2]$, send $[(s_1,s_2,s_2,s_4,s_5,\dots,s_n)]$ to $[(s_1,s_2,s_3,s_4,s_5,\dots,s_n)]=[s]$ and $[(u_1,u_2,u_2,u_4,u_5,\dots,u_n)]$ remains unchanged by lemma \ref{0.6.10}. With  $\F_q(u_1)=\F_{q^m}$ and $[s_2]\neq[u_2]$, send $[(u_1,u_2,u_2,u_4,u_5,\dots,u_n)]$ to $[(u_1,u_2,u_3,u_4,u_5,\dots,u_n)]=[u]$ and $[s]$ remains unchanged by lemma \ref{0.6.10}.
Thus in all cases we can find a tame map $F\in TA_n(\F_q)$ such that $F[r]=[s]$ and $F[u]=[u].$

    \end{proof}
    \begin{lemma}\label{lemma1.10}
    Let $r,u\in\X$ be in different orbits with $u=(u_1,\dots, u_n)$ and $r=(r_1,\dots, r_n)$, $\F_q(r_1)=\F_{q^m}$ and $r\approx u$. Then there exist $G\in TA_n(\F_q)$ such that $G(r)\not \approx G(u)$.
        \end{lemma}

\begin{proof}
Since $r_i\in[u_i]$ $\forall$ $1\leq i\leq n$, in particular $r_1\in[u_1]$. Thus $u_1$ is generator of $\F_{q^m}=\F_q(u_1)$. Define the tame map $G_k=(X_1,\dots,X_k+f_k(X_1),\dots,X_n)$ for $2\leq k\leq n$ where $f_k\in\F_q[X_k]$ is such that $f_k(u_1)=-u_k$ and $f_k([a_1])=0$ if $[a_1]\neq[u_1]$. Define $G=G_2G_3\dots G_n$. Thus $G(u_1,\dots,u_n)=(u_1,0,0,\dots,0)$, $G(r_1,r_2,\dots,r_n)=(r_1,\tilde{r}_2,\tilde{r}_3,\dots,\tilde{r}_n)$. Since $[r]\neq[u]$  thus $G[r]\neq G[u]$. This shows at least one of $[\tilde{r}_i]\neq[0]$ for $i\geq 2$ (as $[r_1]=[u_1]$ given). Hence $G(r)\not \approx G(u)$.
\end{proof}

\begin{proposition} \label{1.11}
Let $s,r,u\in \X$ s.t. $s=(s_1,\dots, s_n)$, $r=(r_1,\dots, r_n)$, and $u=(u_1,\dots,u_n)$  be in different orbits with $\F_{q^m}=\F_q(r_1)=\F_q(s_1)$. Then there exist $F\in TA_n(\F_q)$ s.t. $F([r])=[s]$ and $F([u])=[u]$.
\end{proposition}
\begin{proof}
\underline{Case $r\not\approx u$ and $s \not\approx u$:} is done by lemma \ref{1.61}. \\
\underline{Case $r\approx u$ and $s\not \approx u$:}
In this case $r_i\in[u_i]$ $\forall$ $1\leq i\leq n$, in particular $r_1\in[u_1]$. Thus $u_1$ is generator of $\F_{q^m}=\F_q(u_1)$. Define the tame map $F_k=(X_1,\dots,X_k+f_k(X_1),\dots,X_n)$ for $2\leq k\leq n$ where $f_k\in\F_q[X_k]$ is such that $f_k(u_1)=-u_k$ and $f_k([a_1])=0$ if $[a_1]\neq[u_1]$. Define $F=F_2F_3\dots F_n$. Thus $F(u_1,\dots,u_n)=(u_1,0,0,\dots,0)$, $F(r_1,r_2,\dots,r_n)=(r_1,\tilde{r}_2,\tilde{r}_3,\dots,\tilde{r}_n)$ and $F(s_1,s_2,\dots,s_n)=(s_1,\tilde{s}_2,\tilde{s}_3,\dots,\tilde{s}_n)$. Since $[r]\neq[u]$ and $[s]\neq[u]$, thus $F[r]\neq F[u]$ and $F[s]\neq F[u]$. This shows at least one of $[\tilde{r}_i]\neq[0]$ for $i\geq 2$ (as given $[r_1]=[u_1]$). Similarly at least one of $[\tilde{s}_i]\neq[0]$ for $i\geq 2$ or $[s_1]\neq[u_1]$ . Thus we have reduced this case to the case $r\not_\approx u$ and $s \not_\approx u.$ \\
\underline{Case $r\not \approx u$ and $s \approx u$:} Similar to the previous case.\\
\underline{Case both $r\approx u$ and $s\approx u$:}
Using lemma \ref{lemma1.10} we find a map $G$ such that $G(r)\not \approx G(u)$.
By the previous case applied to the points $G(r), G(s), G(u)$  we find an $H$ such that $H([G(r)])=[G(s)]$ and $H[G(u)]=[G(u)]$. Now taking $F=G^{-1}HG$ we see that $F[r]=G^{-1}H[G(r)]=G^{-1}[G(s)]=[s]$, $F[u]=G^{-1}H[G(u)]=G^{-1}[G(u)]=[u]$.
\end{proof}

The above proposition proves the 2-transitivity of our group \texttt{G} under the assumption  that both $r_1$ and $s_1$ are generators of $\F_{q^m}$.
The remaining part of this section is to prove the 2-transitivity of our group \texttt{G} without any assumption on generators. We will do this using induction steps to the {\em type}:

\begin{definition} Let $s\in \F_{q^m}^n$. We define the type of  $s$ to be the sequence $\overrightarrow{m_s}=(m_1, \ldots, m_n)$ where $[\F_q(s_i):\F_q]=m_i$ and  $lcm(m_1,m_2,\dots,m_n)=m$.
\end{definition}
\begin{definition} The type of $s$ is said to be ordered type if the sequence $\overrightarrow{m_s}=(m_1, \ldots, m_n)$ is decreasing i.e., $m_1\geq m_2\geq\dots\geq m_n.$
\end{definition}
%Most of the time we will assume that $s$ is such that $m_1\geq m_2\geq \ldots\geq m_n$.
Our induction step will involve assuming that we have proven 2-transitivity for all ordered types of
a higher lexicographic order. The case that $m_1=m$ will then be solved by proposition \ref{1.11}.

 The following lemma \ref{1.14} will be helpful to prove the induction step, by making sure that some vector can be assumed to be in a different orbit than another vector without effecting the type of a vector.

\begin{lemma}\label{1.14}
Let $s, u\in \X$ s.t. $s=(s_1,\dots, s_n)$, and $u=(u_1,\dots,u_n)$  be in different orbits and suppose that $s$ has ordered type $\overrightarrow{m_s}=(m_1, \ldots, m_n).$  %is of the type $m_1\geq m_2\geq\dots\geq m_n$ with $m_i=[\F_q(s_i):\F_q]$.
If $[(s_1,\dots,\hat{s_i},\dots,s_n)]=[(u_1,\dots,\hat{u_i},\dots,u_n)]$, then for all $j\neq i$ with $i,j\in\{1,2,\dots,n\}$ there exist a polynomial map $f\in\F_q[(X_1,\dots,\hat{X}_j,\dots,X_n)]$ such that $[(s_1,\dots,\hat{s_i},\dots,s_j+f(s_1,\dots,\hat{s}_j,\dots,s_n),\dots,s_n)]\neq[(u_1,\dots,\hat{u}_i,\dots,u_j,\dots,u_n)]$ and $[\F_q(s_j+f(s_1,\dots,\hat{s}_j,\dots,s_n)):\F_q]=m_j.$
\end{lemma}
\begin{proof}
Since $[(s_1,\dots,\hat{s}_i,\dots,s_n)]=[(u_1,\dots,\hat{u}_i,\dots,u_n)]$ and $[s]\neq [u]$, therefore $(s_1,\dots,\hat{s}_i,\dots,s_n)=\phi(u_1,\dots,\hat{u}_i,\dots,u_n)$, $s_i=\eta(u_i)$ and  $s_i\neq \phi(u_i)$ for $\phi,\eta \in \Delta.$ Thus for any $j\neq i$ we have $[(s_1,\dots,\hat{s}_j,\dots,s_n)]\neq[(u_1,\dots,\hat{u}_j,\dots,u_n)]$. Fix $a=(s_1,\dots,\hat{s}_j,\dots,s_n)$. Define a tame map $F=(X_1,\dots,X_j+f_{a,1}(X_1,\dots,\hat{X}_j,\dots,X_n),\dots,X_n)$ where $f_{a,1}$ is as in lemma \ref{def1.4}.
%(Note that since $s_1$ generates $\F_{q^m}=\F_q(s_1)$, we can indeed apply lemma \ref{def1.4}.)
Thus $F[s]=[(s_1,\dots,s_{j-1},s_j+1,s_{j+1},\dots,s_n)]$ and $F[u]=[u]$. We claim that  $[(s_1,\dots,\hat{s}_i,\dots,s_{j-1},s_j+1,s_{j+1},\dots,s_n)]\neq[(u_1,\dots,\hat{u}_i,\dots,u_n)]$. For suppose that $[(s_1,\dots,\hat{s}_i,\dots,s_{j-1},s_j+1,s_{j+1},\dots,s_n)]=[(u_1,\dots,\hat{u}_i,\dots,u_n)]$. Combining it with $(s_1,\dots,\hat{s}_i,\dots,s_n)=\phi(u_1,\dots,\hat{u}_i,\dots,u_n)$ (in particular $s_k=\phi(u_k)$ for all $k\neq i$) we have $s_j=\phi(u_j)$  and $s_j+1=\phi(u_j)$, thus $1=0$, a contradiction.
%Since $s_k=\phi(u_k)$ for all $k\neq i$, thus $s_j+1=\phi(u_j)$ gives $1=0$, a contradiction.
Also $[\F_q(s_j+f(s_1,\dots,\hat{s_j},\dots,s_n)):\F_q]=[\F_q(s_j+1):\F_q]=m_j.$
\end{proof}

\begin{lemma} \label{1.13}
Let $r=(r_1,\dots,r_n)\in\X$ such that $m_i=[\F_{q}(r_i):\F_q]$ $\forall$ $1\leq i\leq n$ with  $m_1\geq m_2\geq\dots\geq m_n,$ with at least one strict inequality, then there exist $f(X_1,\dots, X_{n-1}) \in \F_q[X_1,\dots, X_{n-1}]$ such that $[\F_q(r_n+f(r_1,\dots, r_{n-1})):\F_q]>m_n$.
 \end{lemma}
\begin{proof}% We can assume that $m_1\geq m_2\geq\dots\geq m_n$ has at least one strict inequality, otherwise $m_1=m_2\dots=m_n=m$ since $lcm\{m_1,\dots, m_{n-1}\}=m$ and this case is done in \ref{1.6}.
 Let $\beta_f=r_n+f(r_1,\dots, r_{n-1})$. Let \[ N=\#\{\beta_f : f(X_1,\dots, X_{n-1})\in\F_q[X_1, \dots, X_{n-1}]\}\]
\[=\#\{f(r_1, \dots, r_{n-1}): f\in\F_q[X_1, \dots, X_{n-1}]\}\]
\[=\#\F_q[r_1, \dots, r_{n-1}]\]
\[=\#\F_{q^{lcm\{m_1,\dots, m_{n-1}\}}}.\]
Since $m_1\geq m_2\geq\dots\geq m_n,$ has at least one strict inequality therefore if $w=lcm\{m_1,\dots, m_{n-1}\}$ then $m_n<w$. Thus
\[\# \{a\in\F_{q^m}|[\F_q(a):\F_q]\leq m_n \}\]
 \[\leq 1+q+\dots+q^{m_n}\] \[<q^w.\]
% This shows that the field $\F_{q^w}$ has at least one more element than the sum of the elements of all subfields of $\F_{q^m_n}$ including the field $\F_{q^m_n}$.
 Thus there exist some $f\in\F_q[X_1, \dots, X_{n-1}]$ such that $[\F_q(\beta_f):\F_q]>m_n.$

\end{proof}
\begin{remark}
If $m_1\geq m_2\geq\dots\geq m_n,$ has no strict inequality, then $m_1=m_2=\dots=m_n=m=lcm(m_1,\dots,m_n).$ Then this is the case done by lemma \ref{1.11}.
\end{remark}
\begin{lemma}\label{123}
Let $s,r,u\in \X$ with $s=(s_1,\dots, s_n), r=(r_1,\dots, r_n)$ and $u=(u_1, \dots, u_n)$ be in different orbits, suppose that $r$ has ordered type $\overrightarrow{m_r}=(m_1, \ldots, m_n).$
 %and suppose that $r_1,\dots,r_n$ is of the type $m_1\geq m_2\geq\dots\geq m_n$ with $m_i=[\F_q(r_i):\F_q]$ for all $1\leq m_i\leq n$.
 Suppose $F_q(s_i)=F_{q^m}$ for some $i\in\{1,2,\dots,n\}$, then there exists $T\in\TA_n(\F_q)$ such that $T(s_1,\dots,s_n)=(s_i,\dots)$, and the ordered type of $r$ remains unchanged under the map $T$.
\end{lemma}
\begin{proof}
If $i=1$, then $T=(X_1,X_2,\dots,X_n)$ identity map. So we can suppose $i\in\{2,3,\dots,n\}.$
We will give the prove of this lemma in two cases.\\
\textbf{Case 1}. If $[(r_2,\dots,r_n)]=[(s_2,\dots,s_n)]$.\\
  This implies that $s_i\in [r_i]$ which means that $\F_q(r_i)=\F_{q^m}$. Thus $m_i=[\F_q(r_i):\F_q]=m$ and so the orderd type of $r$ become $m_1=m_2\dots=m_i=m\geq m_{i+1}\geq\dots\geq m_n$. Taking the map $T=(X_i,X_2,\dots,X_{i-1},X_1,X_{i+1},\dots,X_n)$ will not change the ordered type of $r$ and $T(s_1,\dots,s_n)=(s_i,s_2,\dots,s_{i-1},s_1,s_{i+1},\dots,s_n).$\\
  \textbf{Case 2}. If $[(r_2,\dots,r_n)]\neq[(s_2,\dots,s_n)]$.\\
   Fix $a=(s_2,\dots,s_n)$. Define $T=(X_1+f_{a,s_i-s_1}(X_2,\dots,X_n),X_2,\dots,X_n)$ where $f_{a,s_i-s_1}$ is as in lemma \ref{def1.4} (Notice that since $s_i$ generates $\F_q^m$, we can indeed apply this lemma.)
     Hence $T(s)=(s_i,s_2,\dots,s_{i-1},s_1,s_{i+1},\dots,s_n)$ and $T(r)=r,$ which proves the lemma.

\end{proof}
\begin{lemma}\label{1.1818}
Let $s,r,u\in \X$ with $s=(s_1,\dots, s_n), r=(r_1,\dots, r_n)$ and $u=(u_1, \dots, u_n)$ be in different orbits with $\F_q(s_1)=\F_{q^m}$, then there exists $F\in\TA_n(\F_q)$ such that $F([r])=[s]$ and $F([u])=[u]$.
\end{lemma}
\begin{proof}
We will prove this lemma using mathematical induction on the ordered type of $r$. Our induction step will involve assuming that we have proven the lemma (2-transitivity) for all ordered types $\overrightarrow{\acute{m}_r}=(\acute{m}_1, \acute{m}_2, \ldots, \acute{m}_n)$ of a higher lexicographic order.\\
We may reorder and rename $s_1,s_2,\dots,s_n$ and $r_1,r_2,\dots,r_n$ such that $r_1,r_2,\dots,r_n$ is of ordered type $\overrightarrow{m_r}=(m_1,m_2,\dots,m_n)$ and $s_i$ is generator of field $\F_{q^m}$ (i.e. $s_1$ is moved to the i-th position).
Then by lemma \ref{123} there exist a tame map $T\in\TA_n(\F_q)$ such that $T(s_1,s_2,\dots,s_n)=(s_i,\dots)$ and the ordered type of $r_1,r_2,\dots, r_n$ remains unchanged under $T$.
 Thus the case $m_1=m$ is done by proposition \ref{1.11}. This is the initial induction case.\\
 We will now formulate the second step of induction involving induction hypothesis.
%
%  Assume that we have proven the lemma for elements  $\acute{r}=(\acute{r_1}, \acute{r_2}, \dots, \acute{r_n})\in\X$ of the type $\overrightarrow{\acute{m}}=(\acute{m_1}, \acute{m_2}, \dots, \acute{m_n}) \in \{1, 2, \dots, m\}^n$ with $\overrightarrow{\acute{m}}\ge_{Lex}\overrightarrow{m}$, $\acute{m_i}:=[\F_q(\acute{r_i}) : \F_q]$ , $\acute{m_1}\geq \acute{m_2}\geq\dots\geq \acute{m_n}$, $m=lcm(\acute{m_1}, \acute{m_2}, \dots, \acute{m_n})$ and $\F_{q^m}=\F_q(\acute{r_1}, \acute{r_2}, \dots, \acute{r_n})$.\\
 Define the tame map $G:=(X_1, X_2, \dots, X_n+f(X_1, X_2, \dots, X_{n-1}))$, where $f\in\F_q(X_1, X_2,\dots, X_{n-1})$ is s.t. $m_n+k=[\F_q(\beta):\F_q]$ by lemma \ref{1.13} for some $k\geq1$, where $r_n+f(r_1, r_2, \dots, r_{n-1})=\beta$.
%As a first case suppose that the elements $r$ and $u$ are such that $[(r_1,\dots,r_{n-1})]\neq [(u_1,\dots,u_{n-1})]$.
Fix $a=(u_1,\dots,u_{n-1})$. Define a tame map $H_f:=(X_1,X_2,\dots,X_n+h_{a,-f(a)}(X_1,\dots, X_{n-1}))$ where $h_{a,-f(a)}$ is as in lemma \ref{def1.4}. Define $G_1=H_f\circ G.$\\
\textbf{Case 1} Suppose $[(r_1,\dots, r_{n-1})]\neq [(u_1,\dots, u_{n-1})].$ We compute $G_1([r])=G[r]  =[(r_1, r_2, \dots, r_{n-1}, \beta)]$ and $G_1([u])=H_f\circ G[(u_1, u_2, \dots, u_n)]=H_f[(u_1, u_2, \dots, u_{n-1}, u_n+f(u_1, u_2, \dots, u_{n-1}))]=[(u_1, u_2, \dots, u_{n-1},u_n+f(u_1, u_2, \dots, u_{n-1})-f(u_1, u_2, \dots, u_{n-1}) )]=[(u_1,u_2, \dots, u_{n})]$.  Rearrange $r_1, r_2, \dots, r_{n-1}, \beta$ by a swap map $G_2$ to get its ordered type $\overrightarrow{\acute{m}_r}=(\acute{m}_1,\acute{m}_2,\dots,\acute{m}_n),$ where  $\overrightarrow{\acute{m}_r}\ge_{Lex}\overrightarrow{m_r}$. Then by  the induction argument there exist  a tame map $G_3$ s.t. $G_3(G_2G_1[r])=[s]$ and $G_3G_2G_1[u]=[u]$, hence $F=G_3G_2G_1$.\\
 \textbf{Case 2}.\\
 Now suppose that the elements $r$ and $u$ are such that $[(r_1,\dots,r_{n-1})]=[(u_1,\dots,u_{n-1})]$.
 Then by lemma \ref{1.14} we can find a tame map $F_1$ mapping $[r]$ to $[(\tilde{r_1},\dots,\tilde{r}_n)]$ and $[u]$  to $[u]$ such that $[(\tilde{r_1},\dots,\tilde{r}_{n-1})]\neq[(u_1,\dots,u_{n-1})]$  and $[\F_q(\tilde{r_i}):\F_q]=m_i$ for all $1\leq i\leq n$ (i.e., $r$ and $(\tilde{r_1},\dots,\tilde{r}_{n})$ has same ordered type). Now by applying the case 1 of this lemma we can find a tame map $F_2$ such that $F_2F_1[r]=[s]$ and $F_2F_1[u]=[u],$ hence $F=F_2F_1$ is our required map.
\end{proof}

\begin{proposition} \label{1.725}
Let $s,r,u\in \X$ s.t. $s=(s_1,\dots, s_n)$, $r=(r_1,\dots, r_n)$, and $u=(u_1,\dots,u_n)$  be in different orbits then there exist $F\in TA_n(\F_q)$ s.t. $F([r])=[s]$ and $F([u])=[u]$.
\end{proposition}
\begin{proof}
 Pick $v=(v_1,0,0,\dots,0)$ where $v_1$ is a generator of $\F_{q^m}$. Then by lemma \ref{1.1818} there exist $F_1,F_2\in\TA_n(\F_q)$ such that $F_1([s])=[v]$, $F_1([u])=[u]$ and $F_2([r])=[v]$, $F_2([u])=[u]$. Define $F=F_1^{-1}F_2$. Then $F([r])=[s]$ and $F([u])=[u]$.
\end{proof}

\begin{proof} (of the main theorem \ref{Main}.)
Proposition \ref{1.725} shows that  \texttt{G} acts 2-transitively on the set of orbits $\bar\X$. This implies the primitivity of the group \texttt{G} on the set of orbits $\bar\X.$
Lemma \ref{3-cycles} shows that  \texttt{G} contains 3-cycles.  Hence by Jordan's  theorem \ref{Jordan},  \texttt{G} contains the alternating group $\Alt(\bar{\X})$.\end{proof}

\section{The case when $m$ is prime integer}
\label{Section5a}

Theorem \ref{Main} shows that $\Alt(\bar{\X})\subseteq\texttt{G}$ for $n\geq3.$ In general it is difficult to describe the group $\texttt{G}$ exactly. But for some particular cases we are able to compute $\texttt{G}$ and show which of the two possibilities (alternating or symmetric) it is.
In this subsection we assume that $m$ is prime integer.  We will prove the following proposition:
\begin{proposition}\label{Main1}
Let $m,p$ be prime integers, $q=p^l:$ $l\geq1$ and $n\geq3,$ then $\texttt{G}=Sym(\overline{\X})$ if $q\equiv3,7\mod8$ with $m=2$ and $\texttt{G}=Alt(\overline{\X})$ for all other $q$ and $m$.
\end{proposition}
We will postpone the proof of this proposition to the end of this subsection. We first prove some lemma's that we need.

\begin{lemma}\label{4.7}
 Let $m,p$ are prime integers, $q=p^l:$ $l\geq1$.
Let $F_2:=(X_2,X_1,X_3\dots,X_n)$ then $\pi_{q^m}(F_2)\in\texttt{G}$ is odd permutation for $q=3,7\mod8$ with $m=2$ and is even permutation for all other  $m$ and $q.$
\end{lemma}
\begin{proof}
To see the sign of permutation $\pi_{q^m}(F_2)$ we need to count 2-cycles in the decomposition of $\pi_{q^m}(F_2)$ in transpositions. For this, first we will see how many orbits $F_2$ fixes and then subtract this number from total number of orbits. In this way we will get the total number of orbits moved by $F_2$. Dividing this number by $2$ gives us the number of 2-cycles in $\pi_{q^m}(F_2).$\\
To count the number of orbits fixed by $F_2$, consider $[r]=[F_2(r)]$ for any $r:=(r_1,r_2,\dots,r_n)\in\X$. Then $r=\phi(F_2(r))$ and in particular $r_1=\phi(r_2)$, $r_2=\phi(r_1)$,$r_3=\phi(r_3)$,\dots,$r_n=\phi(r_n)$ for all $\phi\in \Delta.$ For $\phi=id$, we have $r_1=r_2$ and $r_3,r_4,\dots,r_n\in\X$ arbitrary. These are $q^{m(n-1)}-q^{n-1}$ points. For $\phi\neq id$, we have $r_1=\phi^2(r_1)$, $r_2=\phi^2(r_2)$ , $r_3=\phi(r_3),\dots,r_n=\phi(r_n)$ and so $r_1,r_2,r_3,\dots,r_n\in\F_q$ except for $m=2$. When $m=2$ we have $r_1=\phi(r_2)$ arbitrary and so $r_1\in\F_{q^2}\setminus\F_q$ and $r_3,r_4,\dots,r_n\in\F_q.$ Thus the case $\phi\not=id$ gives no point for both $m\neq2$ and $(q^2-q)q^{n-2}$  points satisfy $[r]=[F_2(r)]$ for $m=2.$\\
When counting the number of points which satisfy $[r]=[F_2(r)]$, we get different results for the cases $m\not=2$ and $m=2$, so we distinguish these cases.\\
\textbf{Case 1} Let $m\neq2.$
  In this case the number of points moved by $F_2=(X_2,X_1,X_3,\dots,X_n)$ are $\#\X-(q^{m(n-1)}-q^{n-1})=(q^{mn}-q^n)-(q^{m(n-1)}-q^{n-1})$.
  %As $F_2(r)\neq\phi(r)$ for any $\phi\neq id$, so $F_2(r)\notin[r]$, thus $F_2$ moves every orbit of size $m$ to some different orbit of size $m$ instead of itself (or just fix the whole orbit for the case $F_2(r)=r$).
 %Note that as $F^2=id$ and m is odd, so $F$ moves every orbit of size $m$ to some different orbit of size $m$ instead of itself.
Thus the number of orbits moved by $F_2$ are $\frac{(q^{mn}-q^n)-(q^{m(n-1)}-q^{n-1})}{m}$. So the number of 2-cycles in this permutation are $\frac{(q^{mn}-q^n)-(q^{m(n-1)}-q^{n-1})}{2m}$. Let $Q={(q^{mn}-q^n)-(q^{m(n-1)}-q^{n-1})}.$ Since $q^m\equiv q\mod m,$ thus $Q\equiv0\mod m.$ Also it is trivial to check $Q\equiv0\mod4,$ hence $Q\equiv0\mod4m.$
 % It is easy to see that $Q\equiv0\mod 4m$ in all dimensions except for $q=2$. For $q=2$, $Q\equiv0\mod 4m$ when $n\geq3.$
Thus $\pi_{q^m}(F_2)$ is even.\\
\textbf{Case 2} Let $m=2.$ So the number of points moved by $F_2=(X_2,X_1,\dots,X_n)$ are $\#\X-(q^{2(n-1)}-q^{n-1})-(q^2-q)q^{n-2}=(q^{mn}-q^n)-(q^{m(n-1)}-q^{n-1})-q^n+q^{n-1}=q^{2n}-q^{2(n-1)}-2q^n+2q^{n-1}$.
 % Note that as $F^2=id$ and m is odd, so $F$ moves every orbit of size $m$ to some different orbit of size $m$ instead of itself.
Thus the number of orbits moved by $F$ are $\frac{q^{2n}-q^{2(n-1)}-2q^n+2q^{n-1}}{2}$. So the number of 2-cycles in this permutation are $\frac{q^{2n}-q^{2(n-1)}-2q^n+2q^{n-1}}{4}$. Let $Q_1=q^{2n}-q^{2(n-1)}-2q^n+2q^{n-1}.$ We have
\[ Q_1\mod8= \left\{
  \begin{array}{l l l}

    -4 & \quad \text{if $q\equiv3,7\mod8$}\\
    0  & \quad \text{ if $q\equiv1,5\mod8$, $p=2$ }\\
    \end{array} \right..\]
    This shows $\pi_{q^m}(F_2)$ is odd for $q\equiv3,7\mod8$ and even for $q\equiv1,5\mod8$, $p=2.$
    Thus $\pi_{q^m}(F_2)$ is odd for $q\equiv3,7\mod8$ with $m=2$ and is even for all other values of $q$ and $m$.
\end{proof}
\begin{lemma}\label{4.8}
 Let $m,p$ be prime integers, $q=p^l:$ $l\geq1$ and $\frac{q-1}{2}\notin2\Z$. Let $F_1:=(aX_1,X_2,\dots,X_n):$ $a\in(\F_q)^*.$ Then $\pi_{q^m}(F_1)\in\texttt{G}$ is odd permutation for $q=3,7\mod8$ with $m=2$ and is even permutation for all other  $m$ and $q.$
\end{lemma}
\begin{proof}
\textbf{Case 1} Let $q=2^l:$ $l\geq1.$ Then $sign(F_1)=sign({F_1}^{q-1})=sign(a^{q-1}X_1,X_2,\dots,X_n)=sign(X_1,X_2,\dots,X_n).$ This shows $\pi_{q^m}(F_1)$ is even permutation in this case.\\
\textbf{Case 2} Let $\frac{q-1}{2}$ is odd integer.
Then $sign(F_1)=sign({F_1}^{\frac{q-1}{2}})$.
%Thus to check $sign(F_1)$, consider $\overline{F_1}:={F_1}^{\frac{q-1}{2}}=(-X_1,X_2,\dots,X_n).$
 Thus it is sufficient to consider $\overline{F_1}:={F_1}^{\frac{q-1}{2}}=(-X_1,X_2,\dots,X_n)$ instead of $F_1$ to check its sign.\\
  To see the sign of permutation $\pi_{q^m}(F_1)$ we will proceed similar to lemma \ref{4.7}. First we will see how many orbits $\overline{F_1}$ fixes and then subtract this number from total number of orbits. In this way we will get the total number of orbits moved by $\overline{F_1}$. Dividing this number by $2$ gives us the number of 2-cycles in $\pi_{q^m}(F_2).$\\
 To count the number of orbits fixed by map $\overline{F_1}$, consider $[r]=[\overline{F_1}(r)]$ for $r:=(r_1,r_2,\dots,r_n)\in\X$. Then $\overline{F_1}(r)=\phi(r)$ and in particular $-r_1=\phi(r_1)$, $r_2=\phi(r_2)$,\dots,$r_n=\phi(r_n)$ for all $\phi\in \Delta.$ For $\phi=id$, we have $r_1=0$ and $r_2,r_3,\dots,r_n\in\X$ arbitrary. This gives $q^{m(n-1)}-q^{n-1}$ points fixed by $\overline{F_1}$. For $\phi\neq id$, we have $-r_1=(r_1)^{q^i}$  for some integer $i$ and $r_2,r_3,\dots,r_n\in\F_q$ or equivalently $r_1=0$, ${r_1}^{q^i-1}=-1$ and $r_2,r_3,\dots,r_n\in\F_q$.
  Now we check how many solutions do the equations $(r_1)^{q^i-1}=-1:$ $1\leq i<m$ have in $(\F_{q^m})\setminus(\F_{q})$.
Let $\alpha$ be the generator of $(\F_{q^m})^*$ and let $r_1=\alpha^{\mu_k}$ satisfying $r_1^{q^i-1}=-1$ for some $\mu_k\in\Z$. This gives $\alpha^{\frac{q^m-1}{2}}=-1=\alpha^{\mu_k(q^i-1)}.$ Comparing powers we get $\mu_k=\frac{q^m-1}{2(q^i-1)}+k\frac{q^m-1}{(q^i-1)}:$ $k\in\Z$. We know $i|m$ if and only if $q^i-1|q^m-1$. As $m$ is prime therefore $q^i-1$ does not divides $q^m-1$ for $1<i<m$. So $\mu_k\not\in\Z$. Hence we have no solution for $r_1^{q^i-1}=-1$ when $1<i<m$. For $i=1$ we can write $\mu_k=\frac{q^{m-1}+q^{m-2}+\dots+1}{2}+k(q^{m-1}+q^{m-2}+\dots+1).$ In this case $\mu_k\not\in\Z$ if $m$ is odd prime and $\mu_k\in\Z$ if $m=2.$ The later case $m=2$ gives us $(q-1)$ solutions $r_1=\alpha^{\mu_k},$ satisfying $r_1^{q-1}=-1.$  Thus for $\phi\neq id$, we have $(q-1)q^{n-1}$ points fixed by $\overline{F_1}$ when $m=2$ and no point is fixed by $\overline{F_1}$ when $m$ is odd prime.\\
 For $m$ odd prime, the total number of points in $\X$ moved by $\overline{F_1}$ are  $(q^{mn}-q^n)-(q^{m(n-1)}-q^{n-1}):=Q.$
Thus the number of orbits moved by $F$ are $\frac{Q}{m}$. So the number of 2-cycles in this permutation are $\frac{Q}{2m}$. In this case $Q\equiv0\mod m$ and $Q\equiv0\mod4,$ thus $Q\equiv0\mod4m.$ Hence $\pi_{q^m}(F_1)$ is even permutation in this case.\\
For $m=2,$ the total number of points in $\X$ moved by $\overline{F_1}$ are  $(q^{mn}-q^n)-(q^{m(n-1)}-q^{n-1})-q^n+q^{n-1}=q^{2n}-q^{2(n-1)}-2q^n+2q^{n-1}:=Q_1.$
Thus the number of orbits moved by $F$ are $\frac{Q_1}{2}$. So the number of 2-cycles in this permutation are $\frac{Q_1}{4}$. In this case
\[ Q_1\mod8= \left\{
  \begin{array}{l l l}

    -4 & \quad \text{if $q\equiv3,7\mod8$}\\
    0  & \quad \text{ if $q\equiv1,5\mod8$ }\\
    \end{array} \right..\]
    Thus $F_1$ induces odd permutation for $q\equiv3,7\mod8$ with $m=2$ and even permutation for all other values of $q$ and $m$.\\
\end{proof}
\begin{lemma}\label{4.9}
Let $F=(X_1+f(X_2,X_3,\dots,X_n),X_2,\dots,X_n)$ then $\pi_{q^m}(F)\in\texttt{G}$ is even permutation.
\end{lemma}
\begin{proof} Define $f_{(\alpha_2,\alpha_3,\dots,\alpha_n)}\in\F_q[X_2,X_3,\dots,X_n]$ such that for any $\beta\in(\F_{q^m})^{n-1}\setminus(\F_{q})^{n-1}$
\[ f_{(\alpha_2,\dots,\alpha_n)}(\beta)= \left\{
  \begin{array}{l l l}

    1 & \quad \text{if $\beta=(\alpha_2,\dots,\alpha_n)$}\\
    0  & \quad \text{ if $\beta\neq(\alpha_2,\dots,\alpha_n)$ }\\
    \end{array} \right..\]
Define $\eta=(X_1+f_{(\alpha_2,\alpha_3,\dots,\alpha_n)},X_2,\dots,X_n).$ Then $\eta(a,\alpha_2,\alpha_3,\dots,\alpha_n)=(a+1,\alpha_2,\alpha_3,\dots,\alpha_n):$ $\eta^p=id$, and $\eta(\beta_1,\beta)=0$ for any $\beta\neq(\alpha_2,\alpha_3,\dots,\alpha_n)$, $(\beta_1,\beta)\in\X.$  Thus the order of $\eta$ is $p.$  Hence the permutation induced by $\eta$ is even. Since the maps of the type $\eta$ works as generator for shears,  $\pi_{q^m}(F)\in\texttt{G}$ is even.
\end{proof}
 \begin{proof} (of the main theorem \ref{Main1})
The group $\texttt{G}$ contains the group $Alt(\overline{\X})$ by the theorem \ref{Main}. Since $\TA_n(\F_q)$ is generated by the maps of the form $F_1$, $F_2$ and $F$ as in lemmas \ref{4.7}, \ref{4.8}, \ref{4.9} which are even except when $q\equiv3,7\mod 8$ and $m=2.$ Hence we have our desired result.

\end{proof}

\section{A bound on the index of $\pi_{q^m}(\TA_n(\F_q))$ in $ \pi_{q^m}(\GA_n(\F_q))$}
\label{Section6}

Theorem \ref{Main} tells us that the action of  $\TA_n(\F_q)$, if restricted to $\bar{\X}$, contains the action of $\Alt(\bar{\X})$. However, our goal is to understand $\pi_{q^m}(\TA_n(\F_q))$, and in particular compare it with $\pi_{q^m}(\MA_n(\F_q))\cap \perm(\F_{q^m}^n)$. Define $\MMA_n^m(\F_q):=\pi_{q^m}(\MA_n(\F_q))\cap \perm(\F_{q^m}^n).$ The following theorem estimates how far is the group $\pi_{q^m}(\TA_n(F_q))$ from $\MMA_n^m(\F_q).$
\begin{theorem}\label{main5.3}
 We have the following bound on index $$[\MMA_n^m(\F_q):\pi_{q^m}(\TA_n(F_q))]\leq2^{\sigma_0(m)}\prod_{d|m}d,$$
where $\sigma_0(m)$ is the total number of divisors of $m.$
\end{theorem}
We will postpone the proof of this theorem to the end of this section. The remaining part of this section is devoted towards the preparation of the proof of theorem \ref{main5.3}.
\begin{definition}
% Define $\MMA_n^m(\F_q)=\pi_{q^m}(\MA_n(\F_q))\cap \perm(\F_{q^m}^n)$.
If $S\subseteq \bar{\F}_q^n$, define $\TA_n(\F_q;S)$ as the set of elements in $\TA_n(\F_q)$ which are the identity on $S$.
Similarly, if $S=\Delta S$ (i.e. it is a union of orbits) then define $\pi_S:\MA_n(\F_q)\lp \Maps(S, S)$
and thus also $\pi_S(\TA_n(\F_q)), \pi_S(\TA_n(\F_q;T)$ etc.
\end{definition}

Let $L$ be the union of all orbits of $\Delta$ acting on $\bar{\F}_q^n$ of order a strict divisor of $m$; i.e. $L$ contains
all orbits of size $d$, where $d|m$ but not the ones of size $m$.
Then we have $\pi_L(\TA_n(\F_q))$ as well as $\pi_{q^m}(\TA_n(\F_q;L))$.
The first are permutation on $L$, the second permutations on $\F_{q^m}^n$ which fix $L$. Is there some way to glue these to
get $\pi_{q^m}(\TA_n(\F_q))$? Well, only in part: $\pi_{q^m}(\TA_n(\F_q))$ is not
a semidirect product of the other two, but their sizes compare::

%Let $\X_d=\bigcup^{r_d}_{i=1}\OO_i$ with $\OO_i$ being orbits of size $d$ for all i.
\begin{lemma}\label{5.2}
 \[\#\pi_{q^m}(\TA_n(\F_q))=\#\pi_{{\X}_1}(\TA_n(\F_q))\cdot \#\pi_{q^m}(\TA_n(\F_q;{\X_1})).\]
\end{lemma}

\begin{proof} Pick a representant system $R$ in $\TA_n(\F_q))$ of $\pi_{\X_1}(\TA_n(\F_q))$. Then for each $\pi_{q^m}(F)\in \pi_{q^m}(\TA_n(\F_q))$ there exists a unique $G\in R$ such that $\pi_{q^m}(GF)\in \pi_{q^m}(\TA_n(\F_q;{\X_1}))$.
Thus, $\# R \cdot\#\pi_{q^m}(\TA_n(\F_q;{\X_1}))=\# \pi_{q^m}(\TA_n(\F_q))$.
\end{proof}
%\begin{lemma}
%Define $S^c$ as the complement of a set in $(\F_{q^m})^n$. Let $K<\GA_n(\F_q)$. Then
%\[ \pi_{q^m}(K)=\pi_{\OO_1}(K)\rtimes \pi_{\OO_1^c}(K_{\OO_1}). \]
%\end{lemma}
%
%\begin{proof} Lets write $G=\pi_{q^m}(K), H=\pi_{\OO_1}(K), N=\pi_{\OO_1^c}(K)$, and interpret all these groups as subgroups of $\perm( (\F_{q^m})^n)$.
%The map $K\lp \pi_{\OO_1}(K)$ has $K_{\OO_1}$ as kernel, and thus this is a normal subgroup of $K$.
%This implies that $N$ is a normal subgroup of $G$.
%Furthermore, their intersection is obviously only the identity map.
%Then, any element $\sigma\in G$ fixes the set $\OO_1$ (not the individual elements), and thus there exists a unique element $\tau\in H$ such that $h^{-1}\sigma\in N$, i.e $h^{-1}\sigma=n\in N$ for some unique $n$ and thus $\sigma = hn$ where $h,n$ are unique, satisfying one of the definitions of semidirect product.
%\end{proof}

%The above is of course the most important for $H=\TA_n(\F_q)$.
 The same proof works too to
 \[  \#\pi_{q^m}(\TA_n(\F_q;{\X_1}))=\pi_{\X_2}(\TA_n(\F_q;{\X_1}))\cdot \pi_{q^m}(\TA_n(\F_q;{\X_1\cup \X_2}).\]
Meaning, we can decompose $\#\pi_{q^m}(\TA_n(\F_q))$ into smaller parts: let $d_0:=1,d_1,d_2,\ldots,d_{m-1},d_m:=m $ be the increasing list of divisors of $m$. Let $Q_j=\pi_{\X_{d_{j}}}(\TA_n(\F_q;\bigcup_{i=0}^{i=d_{j-1} }\X_i))$ for all $0\leq j\leq m-1$ and $Q_m=\pi_{q^m}(\TA_n(\F_q;L))$ where  $L$ be the union of orbits of size $d|m, d\not =m$, then\\
$$\#\pi_{q^m}(\TA_n(\F_q))=\#Q_{d_0}\cdot\#Q_{d_1}\dots\cdots\#Q_{d_m}.$$
%$\#\pi_{X_{d_0}}(\TA_n(\F_q))\cdot\#\pi_{\X_{d_1}}(\TA_n(\F_q;{\X_{d_0}}))\cdot\#\pi_{\X_{d_2}}(\TA_n(\F_q;{\X_{d_0}\cup \X_{d_1}}))\dots\#\pi_{\X_{d_{m-1}}}(\TA_n(\F_q;\bigcup_{i=0}^{i=d_{m-2} }\X_i))\cdot\#\pi_{q^m}(\TA_n(\F_q;L)),$\\

The same construction will work for the group $\MMA_n^m(\F_q).$ Meaning, we can decompose $\#\MMA_n^m(\F_q)$ into smaller parts:
$$\#\MMA_n^m(\F_q)=\#G_{d_0}\cdot\#G_{d_1}\dots\cdots\#G_{d_m},$$
where $G_{d_i}$ are defined in corollary \ref{cor3.1}.\\

 Define two subgroups of $Q_d$ by $N_d:=\{m\in Q_d|m(\OO_i)=\OO_i\}$ and $R_d:=\{m\in Q_d|\, m(\{\alpha_1,\dots,\alpha_{r_d}\}) \subseteq \{\alpha_1,\dots,\alpha_{r_d}\}\}$ where $\alpha_i \in\OO_i$ be the fixed representatives of orbits $\OO_i$ of size $d$. The group $N_d$ moves elements inside the orbits $\OO_i$ by keeping orbits $\OO_i$ fixed for all $i$ and the group $R_d$ moves only orbits to orbits. By definition  $N_d=Ker\{H_d\lp \pi_{\overline{\X_{d_{d}}}}(\TA_n(\F_q;\bigcup_{i=0}^{i=d_{d-1} }\overline{\X_i}))\}$ and by theorem \ref{Main} $\Alt\{\alpha_1,\dots,\alpha_{r_d}\}\subseteq R_d\subseteq\Perm\{\alpha_1,\dots,\alpha_{r_d}\}$. Hence the group $N_d$ is normal in group $Q_d.$
 \begin{proposition}
  The group $Q_d$ can be represented as a semidirect product of groups $N_d$ and $R_d$.
  \end{proposition}
  \begin{proof}
 % Let define $$N_d:=\{m\in G_d|m(\OO_i)=\OO_i\}$$ then
% $$N_d=Ker\{G_d\lp \Perm(\bar{X_d})\},$$ where $\bar{X_d}$ is the set of orbits of size $d.$ Thus the group $N_d$ is normal in group $G_d$. Pick $\alpha_i \in\OO_i$ the representatives. Define the group $H_d:=\{m\in G_d|\, m(\{\alpha_1,\dots,\alpha_n\}) \subseteq \{\alpha_1,\dots,\alpha_n\}\}.$ Then $H_d\cong\Perm\{\alpha_1,\dots,\alpha_n\}.$ $N_d\cap H_d=\{id\}$ clear by definition of $N_d$ and $H_d.$
Since $Q_d=\pi_{\X_{d_{d}}}(\TA_n(\F_q;\bigcup_{i=0}^{d_{d-1} }\X_i)),$ if $\tau\in Q_d$ then $\tau(\alpha_i)=\phi^{\mu_i}(\alpha_{\sigma(i)})$ where $\sigma\in\Perm{(r_d)}$, $\phi$ is generator for $\Gal(\F_{q^m}:\F_q)$ and $\mu_i$ is an integer. Then we can find $m_0\in N_d$ and $m_1\in R_d$ such that $m_0(\alpha_i)=\phi^{\mu_i}(\alpha_i)$ and $m_1(\alpha_i)=\alpha_{\sigma(i)}$ and $m_1m_0(\alpha_i)=m_1(\phi^{\mu_i}(\alpha_i))=\phi^{\mu_i}(m_1(\alpha_i))=\phi^{\mu_i}(\alpha_{\sigma(i)}).$ Thus $R_dN_d=Q_d.$ $N_d\cap R_d=\{id\}$ clear by definition of $N_d$ and $R_d.$ Since $N_d$ is normal in $Q_d$,  all conditions for one of the definitions of semidirect product is satisfied.
%Hence $Q_d$ is semidirect product of $N_d$ and $R_d.$
\end{proof}
\begin{corollary}(of the lemma \ref{3-cycles})\label{5.5.4}
 Let $\alpha_i\in\OO_i$ be the representant of orbits of size $d$. For any $i,j\in\N$, there exist a map $\zeta\in N_d$ such that $\zeta(\alpha_i)=\phi^{-1}(\alpha_i)$ and $\zeta(\alpha_j)=\phi(\alpha_j)$ for all $\phi\in\Delta$ and $\zeta(\alpha_k)=\alpha_k$ for all $k\neq i,j.$
\end{corollary}
\begin{proof}
Define $A:=(-1,0,\dots,0,t),B:=(0,-1,0,\dots,0,t),C:=(0,0,\dots,0,t),$ where $t$ is a generator of the multiplicative group $(\F_{q^d})^*.$
Let $\tilde{s}=s_1s_2s_3$ and define $\tilde{w}=\tilde{s}^{-1}s^{-1}d^{-1}\tilde{s}sd$, where \[s_3=(X_1,X_2,X_3+f_3(X_2,X_n),X_4,\dots,X_n),\]\[s_2=(X_1,X_2,\dots,X_{n-1},X_n+f_2(X_2,X_3)),\] \[s_1=(X_1,X_2,X_3+f_1(X_2,X_n),X_4,\dots,X_n),\]
 \[ {f_{3}}(a_2, a_n)= \left\{
  \begin{array}{l l l}

    t & \quad \text{if $(a_2, a_n)=(0,t)$}\\
    0  & \quad \text{elsewhere if $[(a_2, a_n)]\neq[(0,t)]$ }\\
    \end{array} \right.,\]
    \[ {f_{2}}(a_2, a_3)= \left\{
  \begin{array}{l l l}

    \phi(t)-t & \quad \text{if $(a_2, a_3)=(0,t)$}\\
    0  & \quad \text{elsewhere if $[(a_2, a_3)]\neq[(0,t)]$ }\\
    \end{array} \right.,\]
    \[ {f_{1}}(a_2, a_n)= \left\{
  \begin{array}{l l l}

    -t & \quad \text{if $(a_2, a_n)=(0,\phi(t))$}\\
    0  & \quad \text{elsewhere if $[(a_2, a_n)]\neq[(0,\phi(t))]$ }\\
    \end{array} \right.,\]
    and $s,d$ are maps as defined in lemma \ref{3-cycles}.

Note that \[s^{-1}_3=(X_1,X_2,X_3-f_3(X_2,X_n),X_4,\dots,X_n),\]\[s^{-1}_2=(X_1,X_2,\dots,X_{n-1},X_n-f_2(X_2,X_3)),\] \[s^{-1}_1=(X_1,X_2,X_3-f_1(X_2,X_n),X_4,\dots,X_n).\]

    Thus $\tilde{w}(A)=\phi(B)$, $\tilde{w}(B)=C$ and $\tilde{w}(C)=\phi^{-1}(A).$ Define $\eta:=\tilde{w}^{-1}w$, where $w$ is as defined in lemma \ref{3-cycles}.
    Thus $\eta(A)=\tilde{w}^{-1}w(A)=\tilde{w}^{-1}(B)=\phi^{-1}(A)$, $\eta(B)=\tilde{w}^{-1}w(B)=\tilde{w}^{-1}(C)=B$ and $\eta(C)=\tilde{w}^{-1}w(C)=\tilde{w}^{-1}(A)=\phi(C)$.\\
     From lemma \ref{3-cycles}, $s$ permutes only the set $L_1$ and $d$ permutes only the set $L_2$. Then $sd$ (and hence $s^{-1}d^{-1}$) acts trivially on $\bar{\X_d}\setminus(L_1\bigcup L_2)$ and nontrivially only on a subset of $L_1\bigcup L_2$. Thus the map $\tilde{w}$  acts trivially on $\bar{\X_d}\setminus(L_1\bigcup L_2)$ and nontrivially only on a subset of $L_1\bigcup L_2$. Since the map $w$ fixes all orbits except $[A]$, $[B]$ and $[C]$, hence the map $\tilde{w}$ can only acts nontrivially on orbits $[A]$, $[B]$ and $[C]$ due to definition of $\tilde{s}$.\\
Since the group $Q_d$ contains the alternating permutation group on $\{\alpha_1,\dots,\alpha_{r_d}\}$ by theorem \ref{Main},  it is 2-transitive. Thus there exist a map $\xi\in Q_d$ such that $\xi(A)=\alpha_i$ and $\xi(C)=\alpha_j.$ Hence $\zeta=\xi\eta\xi^{-1}$ is our required map.
\end{proof}
\begin{corollary} The group $N_d$ contains a group isomorphic to $\{(a_1,a_2,\dots,a_{r_d})\in(\Z/\Z)^{r_d};a_1+a_2+\dots+a_{r_d}=0\}.$
\end{corollary}
\begin{proof} Since for any $\sigma\in N_d$,  $\sigma(\alpha_i)=\phi^{a_i}(\alpha_i)$ and $a_1+a_2+\dots+a_{r_d}=0$ by corollary \ref{5.5.4}, where $\alpha_i$ are representative for orbits $\OO_i$ of order $d$ and $\phi$ is generator for $Gal(\F_{q^m}:\F_q)$. Since $\Gal(\F_{q^d}:\F_q)\cong(\Z/d\Z)$, hence $\{(a_1,a_2,\dots,a_{r_d})\in(\Z/\Z)^{r_d};a_1+a_2+\dots+a_{r_d}=0\}\subseteq N_d.$
\end{proof}
\begin{lemma}\label{5.7} We have the following bound on the index $$[G_d:Q_d]\leq2d.$$
\end{lemma}
\begin{proof} Since $$G_d\cong (\Z/d\Z)^{r_d}\rtimes (\Perm(r_d)),$$ $$Q_d\cong N_d\rtimes R_d.$$  Also by lemma \ref{5.5.4} $$\{(a_1,a_2,\dots,a_{r_d})\in(\Z/\Z)^{r_d};a_1+a_2+\dots+a_{r_d}=0\}\subseteq N_d$$ and by theorem \ref{Main}
$$\Alt\{\alpha_1,\dots,\alpha_{r_d}\}\subseteq R_d\subseteq\Perm\{\alpha_1,\dots,\alpha_{r_d}\}.$$
Since $[(\Z/d\Z)^{r_d}:N_d]\leq d,$ $[\Perm(r_d):R_d]\leq2,$ hence $[G_d:Q_d]\leq2d.$
\end{proof}

\begin{proof} (of main theorem \ref{main5.3}) Since
 $$\#\pi_{q^m}(\TA_n(\F_q))=\#Q_{d_0}\cdot\#Q_{d_1}\cdots\#Q_{d_m},$$
$$\#\MMA_n^m(\F_q)=\#G_{d_0}\cdot\#G_{d_1}\cdots\#G_{d_m}.$$
By lemma \ref{5.7} we have the result.
\end{proof}
\begin{corollary} (of the theorems \ref{Main1}, \ref{main5.3})\\
For $m=2,q\cong3,7\mod8$ we have $$[\MMA_n^2(\F_q):\pi_{q^2}(\TA_n(F_q))]\leq 2.$$
\end{corollary}
\begin{proof}
Combine theorems \ref{Main1} and \ref{main5.3}.
\end{proof}

\end{document}